\documentclass[a4paper,11pt, reqno, 
]{amsart}

\usepackage[english]{babel}
\usepackage[utf8]{inputenc}

\usepackage{graphicx}
\usepackage{amscd}
\usepackage{amsfonts}
\usepackage{amsaddr}
\usepackage{latexsym}
\usepackage{amsthm}
\usepackage{amssymb,amsmath}
\usepackage{enumerate}
\usepackage{verbatim}

\usepackage{import}

\usepackage[titletoc,title]{appendix}

\usepackage[usenames]{color}

 \usepackage[colorlinks=true]{hyperref}
\hypersetup{urlcolor=blue, citecolor=red}

\usepackage{subfigure}

\usepackage{pgfplots}
\tikzset{ font={\fontsize{9pt}{12}\selectfont}}

\title[Rat. maps with {F}atou comp. of arbitrarily large connectivity]{\bf Rational maps with {F}atou components of arbitrarily large connectivity}

\author[Jordi Canela]{Jordi Canela}
\thanks{The  author was partially supported by ANR project Lambda, ANR-13-BS01-002.}
\address{Institut de Mathématiques de Toulouse\\
Université Paul Sabatier,\\
118, route de Narbonne, F-31062 Toulouse, France}
\email{Jordi.Canela\_Sanchez@math.univ-toulouse.fr}

\newtheorem{teor}{Theorem} [section]
\newtheorem{thm}[teor]{Theorem}
\newtheorem{propo}[teor]{Proposition}

\newtheorem{lemma}[teor]{Lemma}
\newtheorem{co}[teor]{Corollary}

\newtheorem*{teoremB}{Theorem B}
\newtheorem*{teoremA}{Theorem A}

\theoremstyle{definition}
\newtheorem{defin}[teor]{Definition}
\newtheorem{defi}[teor]{Definition}
\newtheorem{rem}[teor]{Remark}

\newsavebox{\savepar}

\newcommand{\com}{\mathbb{C}}
\newcommand{\wcom}{\widehat{\mathbb{C}}}

\newcommand{\real}{\mathbb{R}}

\newcommand{\nat}{\mathbb{N}}

\newcommand{\dis}{\mathbb{D}}
\newcommand{\cercle}{\mathbb{S}^1}

\makeatletter
\def\blfootnote{\gdef\@thefnmark{}\@footnotetext}
\makeatother

\date{\today}

    {\end{list}}

\begin{document}
\begin{abstract}
 {\noindent \small We study the family of singular perturbations of Blaschke products \linebreak $B_{a,\lambda}(z)=z^3\frac{z-a}{1-\overline{a}z}+\frac{\lambda}{z^2}$. We analyse how the connectivity of the Fatou components varies as we move continuously the parameter $\lambda$. We prove that all possible escaping configurations of the critical point $c_-(a,\lambda)$ take place within the parameter space. In particular, we prove that there are maps $B_{a,\lambda}$ which have Fatou components of arbitrarily large finite connectivity within their dynamical planes.}
\end{abstract}

\maketitle



\section{Introduction}

Given a rational map $f:\wcom\rightarrow \wcom$, where $\wcom=\com\cup\{0\} $ denotes the Riemann Sphere, we consider the discrete dynamical system provided by the iterates of $f$. This dynamical system splits $\wcom$ into two totally invariant sets, the Fatou set $\mathcal{F}(f)$, which is defined  as the set of points $z\in\wcom$ such that the family $\{f^n, n\in\com\}$ is normal in some neighbourhood of $z$,  and its complement, the Julia set $\mathcal{J}(f)$. The dynamics of the points $z\in\mathcal{F}(f)$ is stable in the sense of normality whereas the dynamics of the points $z\in\mathcal{J}(f)$ presents a chaotic behaviour. The Fatou set $\mathcal{F}(f)$ is open and, hence, $\mathcal{J}(f)$ is closed. The connected components of $\mathcal{F}(f)$ are called Fatou components and are mapped under $f$ among themselves. A Fatou component $U$ is called periodic if there exists $q\in\nat$ with $f^q(U)=U$ and preperiodic if there exists $q\in\nat$ such that $f^q(U)$ is periodic. All Fatou components of rational maps are either periodic or preperiodic (see \cite{Su}). Moreover, any cycle of periodic Fatou components of a rational map has at least a critical point, i.e.\ a point $z\in\wcom$ such that $f'(z)=0$, related to it. For a more detailed introduction to the dynamics of rational maps we refer to \cite{Bear} and \cite{Mi1}. 

The connectivity of a domain $D\subset\wcom$ is given by the number of connected components of its boundary $\partial D$. It is well known that periodic Fatou components have connectivity 1, 2 or $\infty$ (see \cite{Bear}). However, preperiodic Fatou components may have finite connectivity greater that 2. Beardon \cite{Bear} introduced an example suggested by Shishikura of a family of rational maps with Fatou components of finite connectivity greater than 2. Baker, Kotus and L{\"u} \cite{BKL} proved that, given any $n\in\nat$, there exists rational and meromorphic transcendental maps with preperiodic Fatou components of connectivity $n$ by means of a quasiconformal surgery procedure. Later on, Qiao and Gao \cite{QJ}, and Stiemer \cite{Sti} provided explicit examples of families of rational maps with such dynamical properties. However, the examples presented in \cite{BKL}, \cite{QJ} and \cite{Sti} use an increasing number of critical points. Consequently, the degree of these rational maps grows with $n$. 

In \cite{Can1} we introduced a family   of singularly perturbed Blaschke products (see Equation~(\ref{eqblasperturbed})) whose maps have, under certain dynamical conditions,  Fatou components of arbitrarily large finite connectivity (see Theorem~\ref{thmavell}). We also provided numerical evidence showing that parameters satisfying these conditions actually exist. The main goal in this paper is to give a rigorous proof of this fact, that is, to show that this family of perturbations of Blaschke products contains examples of rational maps having  Fatou components of arbitrarily large finite connectivity (in the same dynamical plane).  To our knowledge, there are no previous examples known to show this phenomenon.

The study of singular perturbations in holomorphic dynamics was introduced by  McMullen \cite{McM1} to prove the existence of Julia sets with buried components, i.e.\ connected components of $\mathcal{J}(f)$ which do not lie in the boundary of any Fatou component, that are Jordan curves. He considered the singular  perturbations of $z^n$ given by  

\begin{equation}\label{perturbedpolyn}
Q_{\lambda, n, d}(z)=z^n+\frac{\lambda}{z^d},
\end{equation}

\noindent where $\lambda\in\com^*=\com\setminus\{0\}$ and $n,d\in\nat$. These maps have $z=\infty$ as a superattracting fixed point. Moreover, they have $n+d$ critical points which appear symmetrically around the pole $z=0$. McMullen showed that, if $1/n+1/d<1$ and $|\lambda|$ is small enough then the Julia set consists of a Cantor set of quasicircles. A quasicircle is a Jordan curve which is the image of the unit circle under a quasiconformal map (see e.g.\ \cite{BF}). Devaney, Look and Uminsky \cite{DLU} studied the Julia sets of the maps $Q_{\lambda,n,d}$ from a more general point of view. They showed that if the critical orbits converge to $\infty$ then $\mathcal{J}(Q_{\lambda,n,d})$ is either a Cantor set of points, or a Cantor set of quasicircles, or a Sierpinski curve (a homeomorphic image of a Sierpinski carpet). Afterwards, the family $Q_{\lambda,n,d}$ (see e.g.\  \cite{Mora}, \cite{QWY}, \cite{QRWY}) and singular perturbations of polynomials other than $z^d$ (see e.g.\ \cite{BDGR} and \cite{GMR}) have been the object of study of several papers.

In \cite{Can1} we investigate singular perturbations of the Blaschke products 
\begin{equation}\label{eqblas}
B_a(z)=z^3\frac {z-a}{1-\overline{a}z}
\end{equation}
where $a\in\dis^*=\dis\setminus\{0\}$. These Blaschke products have $z=0$ and $z=\infty$ as superattracting fixed points of local degree 3. Moreover,  their Julia set, which is the common boundary of the basins of attraction of $z=0$ and $z=\infty$, is given by $\mathcal{J}(B_a)=\cercle$. In this sense, the dynamics of  $B_a$ is very similar to the one of the polynomials $z^3$. However, the maps $B_a$ have two extra critical points $c_-(a)\in\dis$ and $c_+(a)\in\com\setminus\overline{\dis}$. We consider the singular perturbations given by

\begin{equation}\label{eqblasperturbed}
B_{a,\lambda}(z)=z^3\frac{z-a}{1-\overline{a}z}+\frac{\lambda}{z^2}
\end{equation}

\noindent where $a\in\dis^*$ and $\lambda\in\com^*$. These singularly perturbed maps have two critical points $c_-(a,\lambda)$ and $c_+(a,\lambda)$ which are analytic continuation of the ones of $B_a$. If $|\lambda|$ is small, the dynamics of $B_{a,\lambda}$ is similar to the one of $Q_{\lambda,3,2}$ in a neighbourhood of $z=0$ and we obtain a McMullen-like Julia set. 
However, the extra critical point $c_-$ allows the maps $B_{a,\lambda}$ to have richer dynamics than the maps $Q_{\lambda,3,2}$. 
The main theorem in \cite{Can1} describes the connectivity of the Fatou components of $B_{a,\lambda}$ in the case that $|\lambda|$ is small enough and $c_-(a,\lambda)$ belongs to the basin of attraction of $z=\infty$, $A_{a,\lambda}(\infty)$. We denote the immediate basin of attraction of $\infty$, i.e.\ the connected component of  $A_{a,\lambda}(\infty)$ which contains $\infty$, by  $A^*_{a,\lambda}(\infty)$. We denote by $\dis_{R}$ the disk centred at $0$ with radius $R$ and by $\dis_{R}^*$ the punctured disk $\dis_{R}\setminus\{0\}$.  
  
\begin{thm}[{\cite[Theorem A]{Can1}}]\label{thmavell}

Fix $a\in\dis^*$. There exists a constant $\mathcal{C}(a)$ such that if $\lambda\in\dis_{{C}(a)}^*$ and $c_-(a,\lambda)\in A_{a,\lambda}(\infty)$, then $c_-(a,\lambda)$ belongs to a connected component $\mathcal{U}_c$ of $A_{a,\lambda}(\infty)\setminus A_{a,\lambda}^*(\infty)$ and exactly one of the following holds.

\begin{enumerate}[a)]
\item The Fatou component $\mathcal{U}_c$ is simply connected. All Fatou components of $B_{a, \lambda}$ have connectivity 1 or 2. (see Figure~\ref{dynamfigureABC} (a) and (b)).
\item The Fatou component $\mathcal{U}_c$ is multiply connected and does not surround $z=0$. All Fatou components of $B_{a, \lambda}$ have connectivity 1, 2 or 3. (see Figure~\ref{dynamfigureABC} (c) and (d)).
\item The Fatou component $\mathcal{U}_c$ is multiply connected and surrounds $z=0$. All Fatou components of $B_{a, \lambda}$ have finite connectivity but there are components of arbitrarily large connectivity. (see Figure~\ref{dynamfigureABC} (e) and (f)).
\end{enumerate} 
\end{thm}

\begin{figure}[p]
    \centering
    
    \subfigure[\scriptsize{$B_{0.5i,\lambda}$ for $\lambda=-1.9\times 10^{-6}+3.15\times 10^{-5}i$}]{
    \includegraphics[width=160pt]{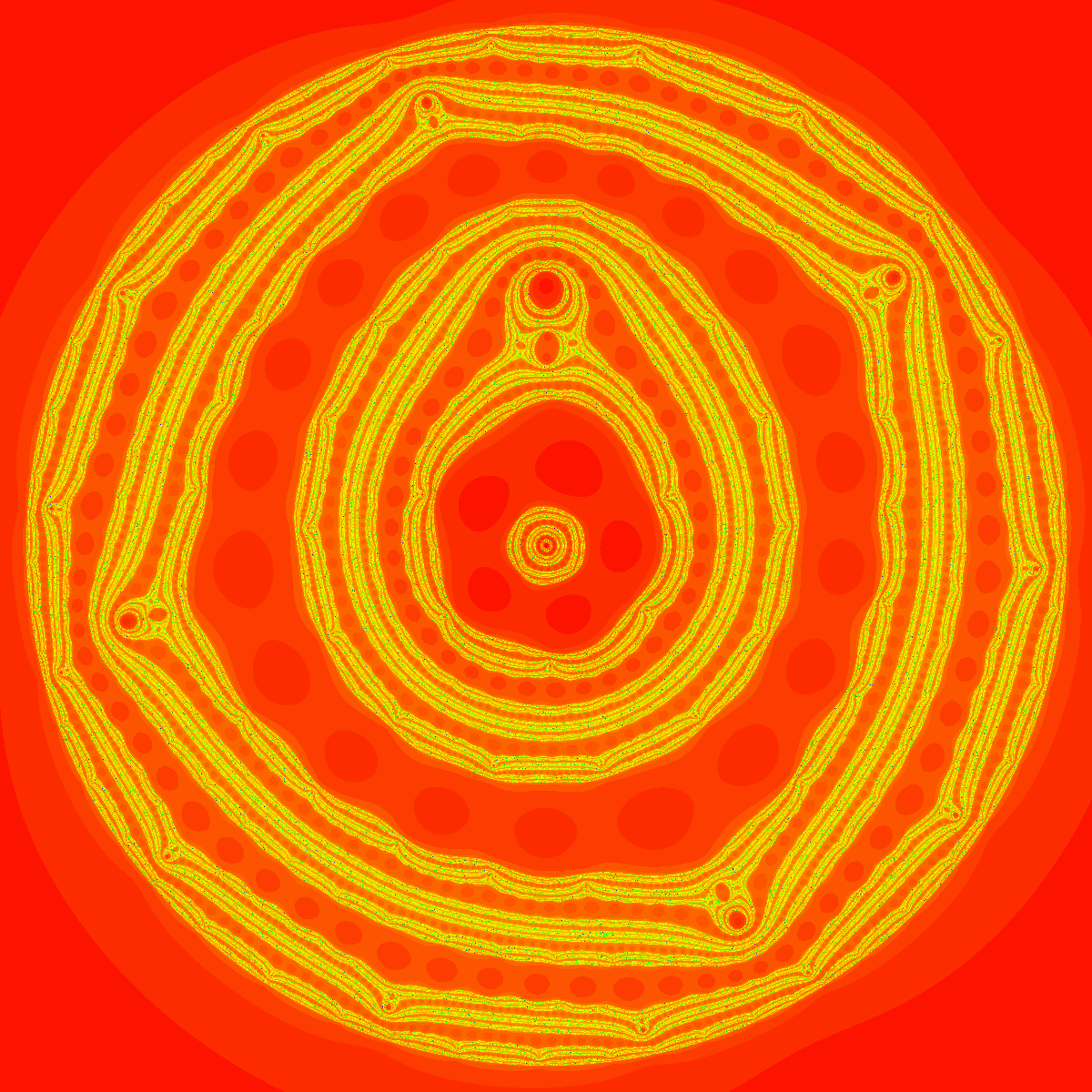}}
    \hspace{0.1in}
    \subfigure[\scriptsize{Zoom in Figure (a) } ]{
    \def\svgwidth{160pt}
    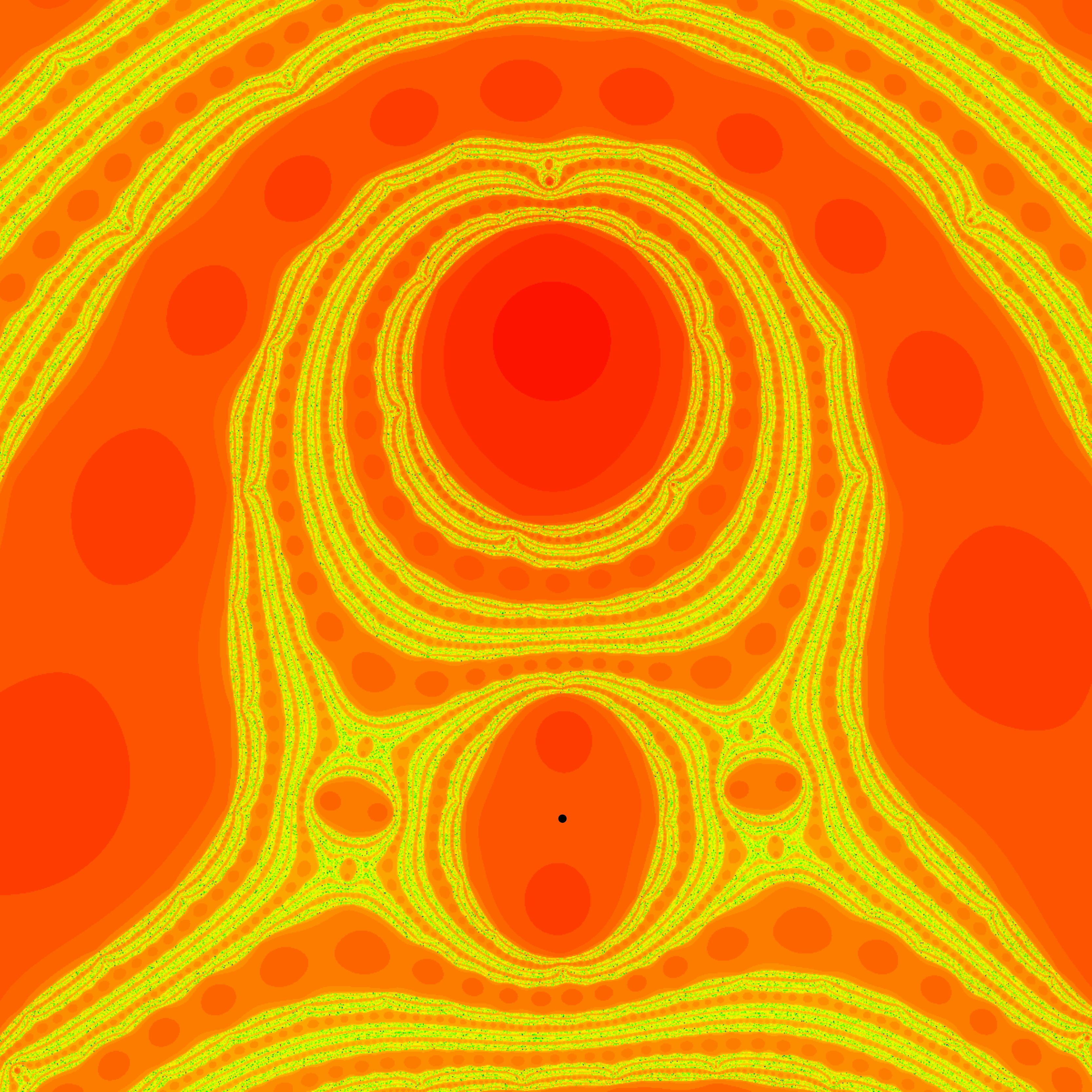}
    
        \subfigure[\scriptsize{$B_{0.5i,\lambda}$ for $\lambda=9.5\times 10^{-7}+3.05\times 10^{-5}i$}]{
    \includegraphics[width=160pt]{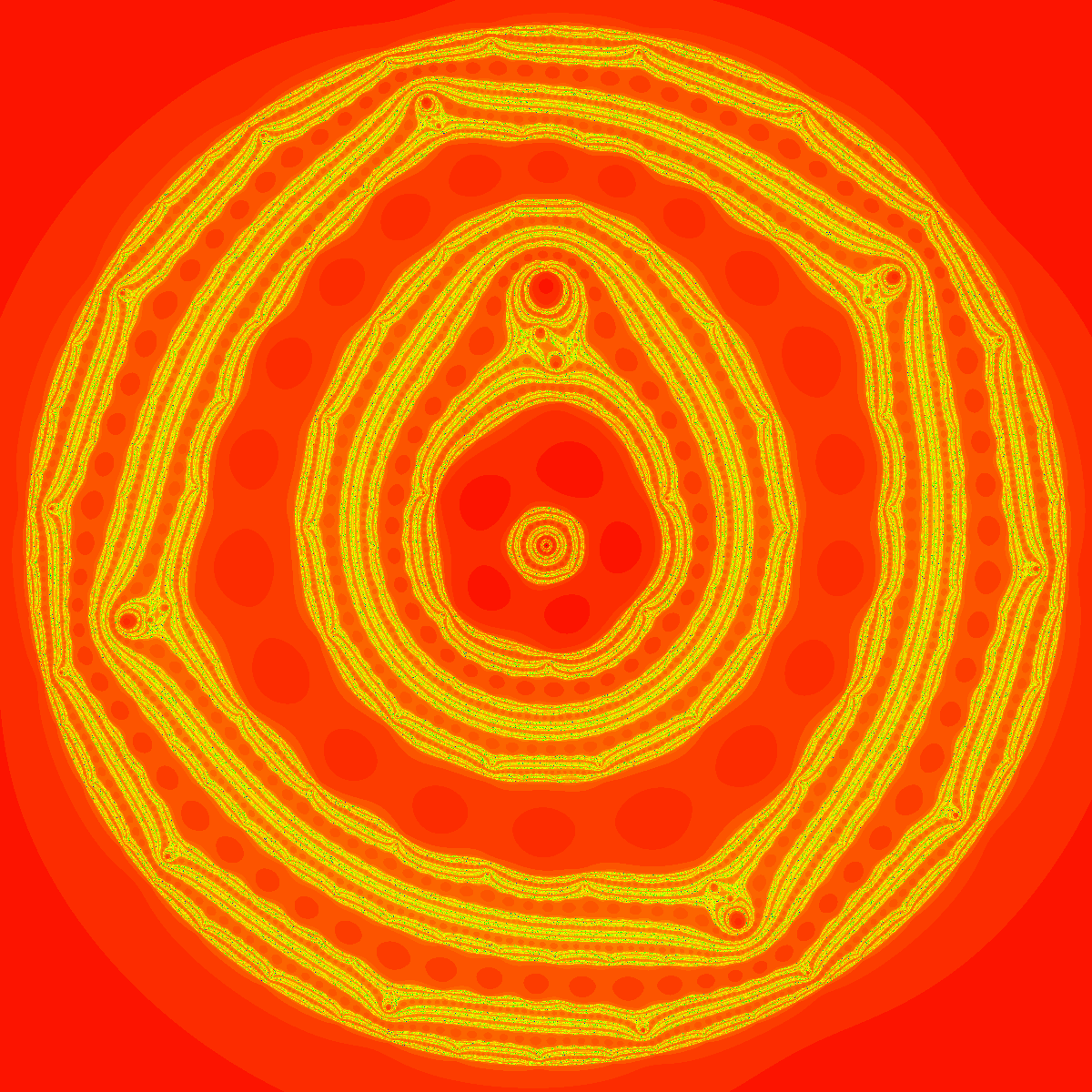}}
    \hspace{0.1in}
        \subfigure[\scriptsize{Zoom in Figure (b) } ]{
    \def\svgwidth{160pt}
	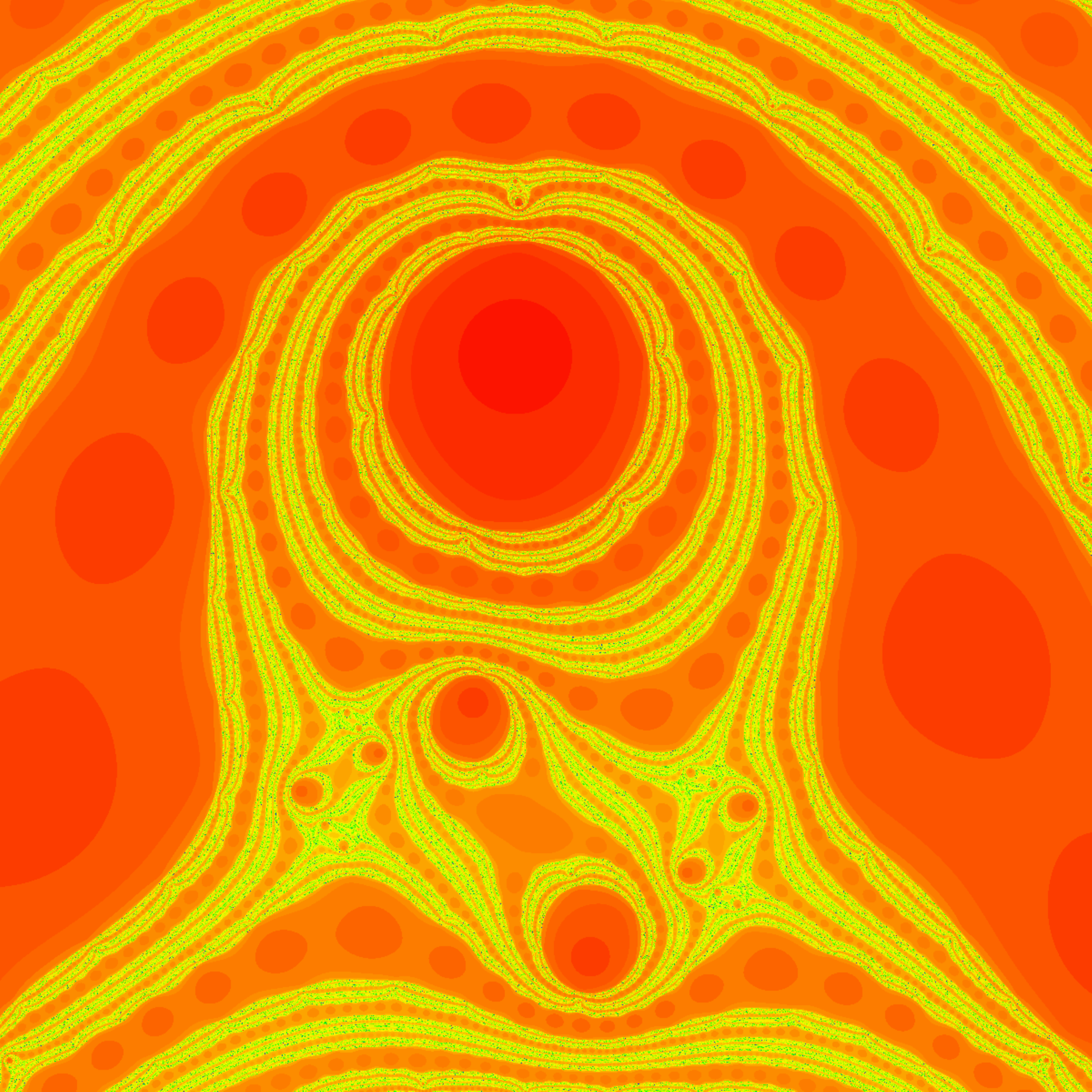}
	
        \subfigure[\scriptsize{$B_{0.5i,\lambda}$ for $\lambda=7.74\times 10^{-6}+9.9\times 10^{-6}$}]{
    \includegraphics[width=160pt]{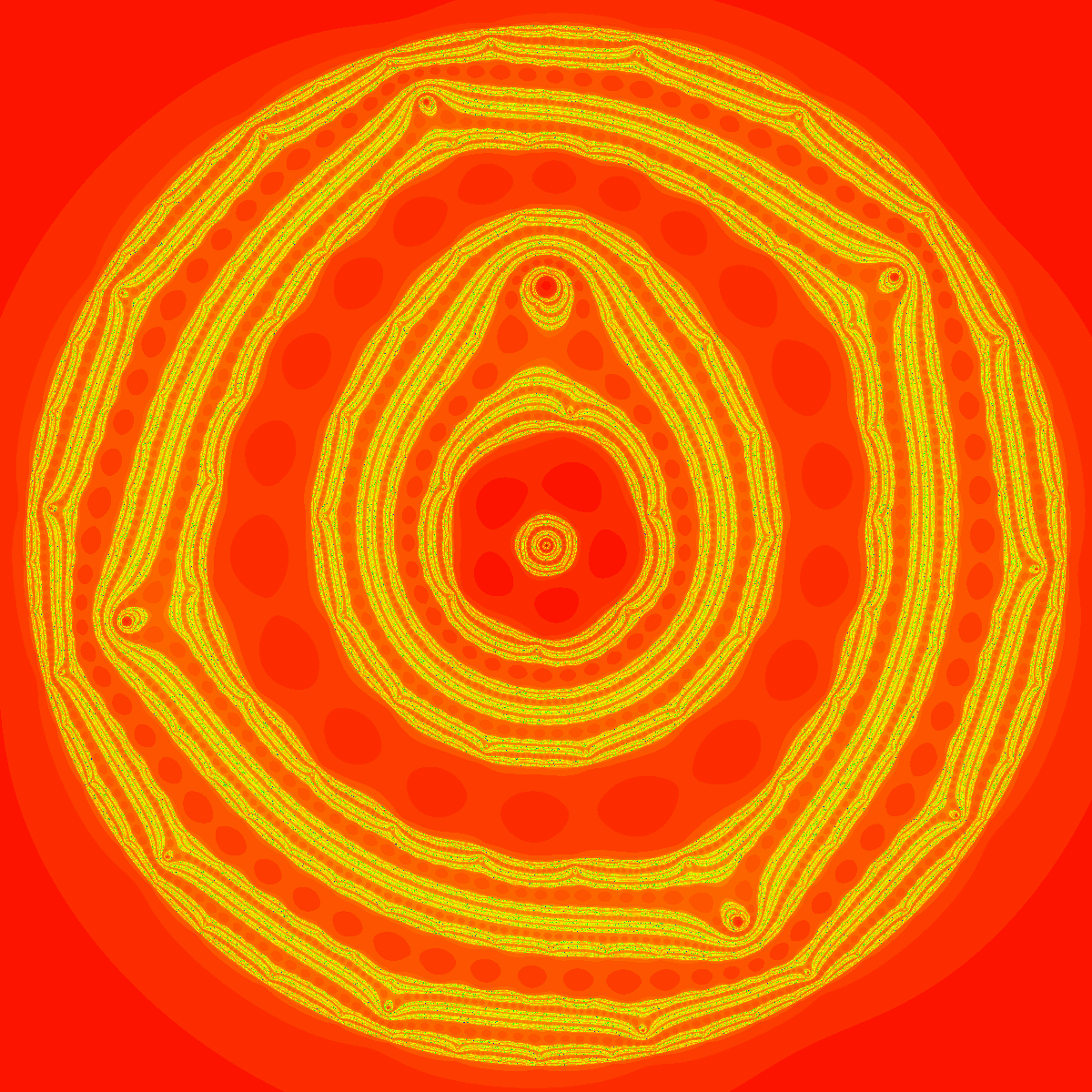}}
    \hspace{0.1in}
    \subfigure[\scriptsize{Zoom in Figure (c) } ]{
    \def\svgwidth{160pt}
    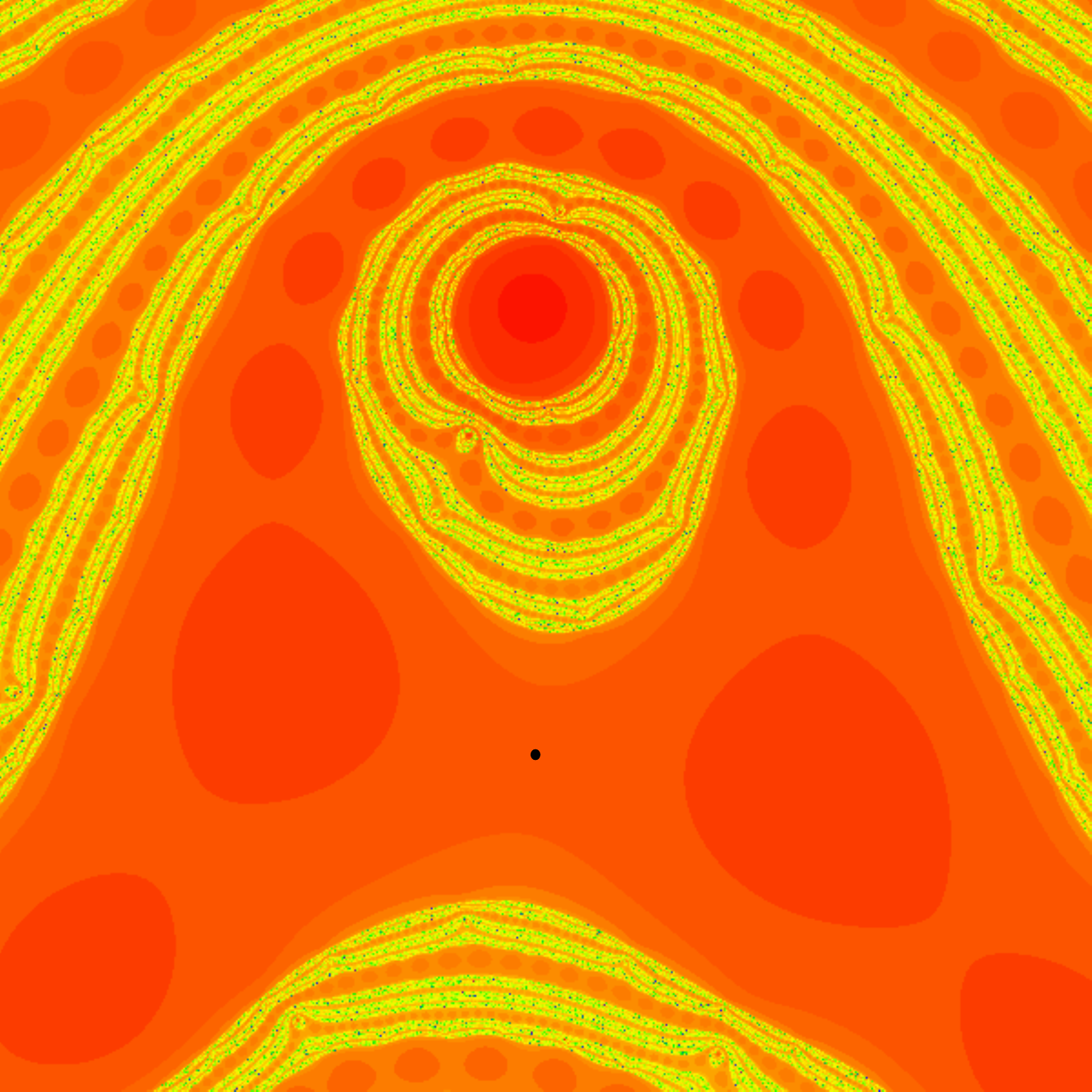}

	 \caption{\small Dynamical planes of 3 maps $B_{a,\lambda}$ with $a=0.5i$ for which the 3 statements of Theorem~\ref{thmavell} hold. We use a scaling from yellow to red to plot the basin of attraction of $z=\infty$. We may observe an approximation of the Julia set in yellow. }

    \label{dynamfigureABC}
\end{figure}

It follows from Theorem~\ref{thmavell} that there may be parameters for which the maps $B_{a,\lambda}$ have Fatou components of arbitrarily large finite connectivity (within a single dynamical plane).  The main goal of this paper is to prove the existence of parameters $(a,\lambda)\in\dis^*\times\com^*$ such that $|\lambda|<\mathcal{C}(a)$ and $c_-(a,\lambda)\in A_{a,\lambda}(\infty)$, realizing each of the three statements of Theorem~\ref{thmavell}. This is the content of Theorem~A.

\begin{teoremA}
There exists parameters $(a_i,\lambda_i)\in\dis^*\times\com^*$, $i=1,2,3$, such that $|\lambda_i|<\mathcal{C}(a_i)$,  $c_-(a_i,\lambda_i)\in A_{a_i,\lambda_i}(\infty)$, and each of the following hold.
\begin{enumerate}[a)]
\item All Fatou components of $B_{a_1, \lambda_1}$ have connectivity 1 or 2.
\item  All Fatou components of $B_{a_2, \lambda_2}$ have connectivity 1, 2 or 3.
\item  All Fatou components of $B_{a_3, \lambda_3}$ have finite connectivity but there are components of arbitrarily large connectivity. 
\end{enumerate} 

\end{teoremA}


The rational maps $B_{a,\lambda}$ have 7 free critical orbits which are not related by any kind of symmetry. Because of this, it is very difficult to investigate the parameter space of $B_{a,\lambda}$ from a global point of view. However,  if $\lambda\in\dis_{{C}(a)}^*$ then all critical points but, maybe, $c_-(a,\lambda)$ belong to the basin of attraction of $\infty$ (see Theorem~\ref{thmcritzeros}). Hence, fixed $a_0\in\dis^*$, we can draw the parameter plane of $B_{a_0,\lambda}$ for $|\lambda|$ small using the orbit of $c_-(a_0,\lambda)$ (see Figure~\ref{param05i}). These numerical experiments show annular hyperbolic components of parameters which surround $\lambda=0$ and accumulate on it (see Figure~\ref{param05i} (left)). These hyperbolic components correspond to parameters for which statement c) of Theorem A holds. In our next result, Theorem~B, we show that there are multiply connected hyperbolic components which surround $\lambda=0$ and accumulate on it. We conjecture that these hyperbolic components are annuli.

\begin{figure}[hbt!]
\centering
    \subfigure{
    \begin{tikzpicture}
    \begin{axis}[width=7.7cm,  axis equal image, scale only axis,  enlargelimits=false, axis on top]
      \addplot graphics[xmin=-0.000065,xmax=0.00008,ymin=-0.0000725,ymax=0.0000725] {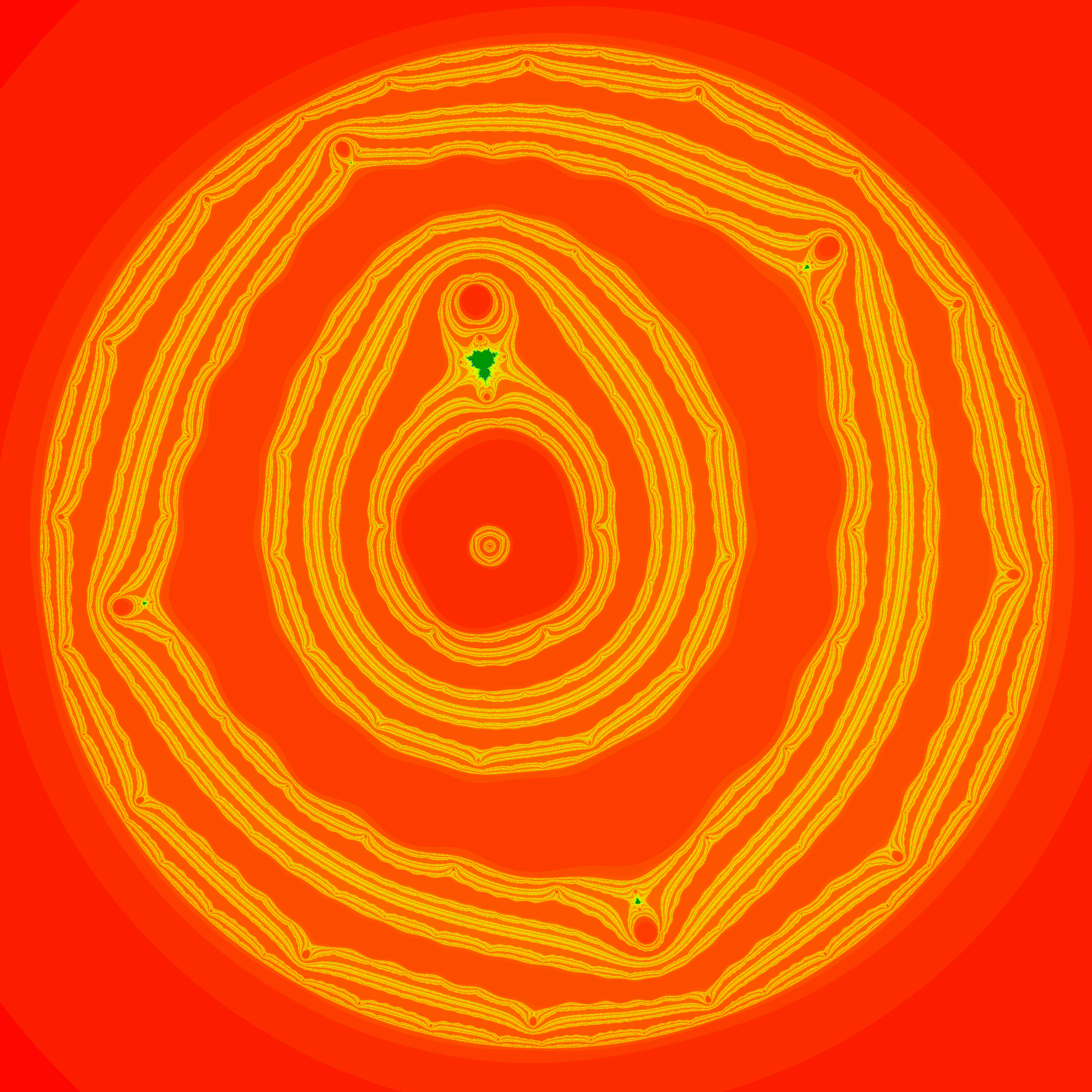};
    \end{axis}
  \end{tikzpicture}}
     \subfigure{
     \begin{tikzpicture}
     \begin{axis}[width=7.7cm,  axis equal image, scale only axis,  enlargelimits=false, axis on top]
      \addplot graphics[xmin=-0.000012,xmax=0.000008,ymin=0.000017,ymax=0.000037] {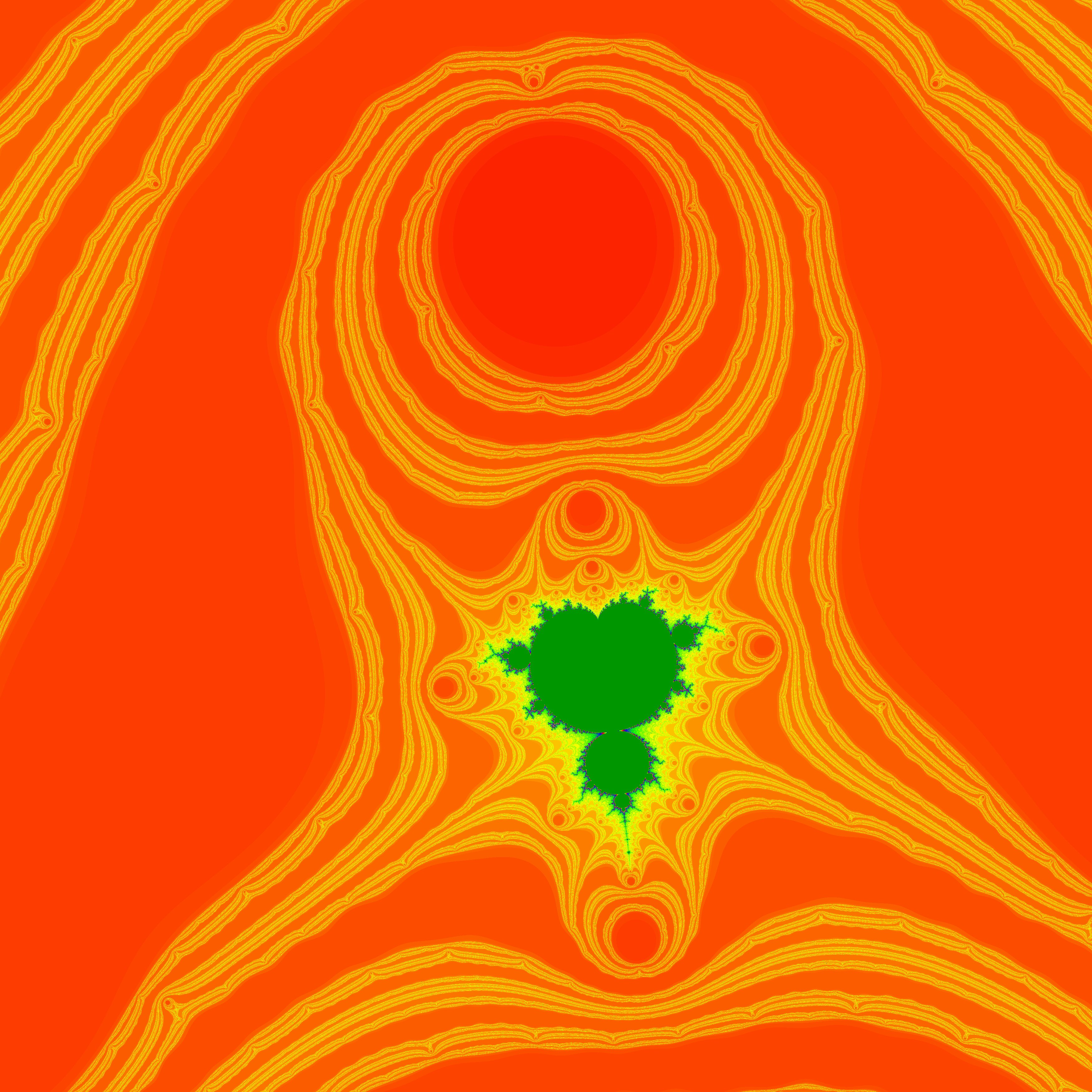};
    \end{axis}
  \end{tikzpicture}}
    \caption{\small{Parameter plane of $B_{a,\lambda}$ for $a=0.5i$ fixed. The $x$-axis corresponds to $\rm{Re}(\lambda)$ and the $y$-axis to $\rm{Im}(\lambda)$. The parameters for the left figure corresponds to $\rm{Re}(\lambda)\in (-6.5\times 10^{-5}, 8\times 10^{-5} )$ and $\rm{Im}(\lambda)\in (-7.25\times 10^{-5}, 7.25\times 10^{-5} )$. The right figure is a zoom in the left one. Colours are as follows. We use a scaling from yellow to red to draw the parameters for which $c_-\in A(\infty)$ and green for $c_-\notin A(\infty)$. }}
    \label{param05i}
\end{figure}

\begin{teoremB}
Fixed $a\in\dis^*$, there are multiply connected hyperbolic components which surround $\lambda=0$ and accumulate on it. The corresponding maps $B_{a,\lambda}$ have Fatou components of arbitrarily large finite connectivity.
\end{teoremB}

These numerical experiments also show simply connected hyperbolic components which correspond to parameters for which statement a) of Theorem~A holds. They are surrounded by annular hyperbolic components which correspond to parameters for which statement b) of Theorem~A holds (see Figure~\ref{param05i} (right)). In Theorem~\ref{thmcaseB} we prove that any parameter $\lambda$ for which statement a) holds is surrounded by multiply connected hyperbolic components of parameters for which statement b) holds.

The paper is structured as follows. In Section~\ref{prelim} we introduce the main dynamical properties of the maps $B_{a,\lambda}$. Afterwards, in Section~\ref{annulardyn} we study the dynamics of the preimages of the annular Fatou component which contains the 5 critical points and zeros which appear after the singular perturbation, and analyse their dependence on $\lambda$. Finally, in Section~\ref{sectionThmAB} we prove Theorem~A and Theorem~B.

\textit{Acknowledgements.} The author would like to thank A.~Cheritat and P.~Roesch for helpful discussions, and N.~Fagella and A.~Garijo for  useful comments.

\section{Preliminaries on the dynamics of the singular perturbations}\label{prelim}


The goal of this section is to describe the dynamics of the singularly perturbed Blaschke products $B_{a,\lambda}(z)$ (Equation (\ref{eqblasperturbed})). We introduce the results proven in \cite{Can1} which are used along the paper. We first explain the dynamics of the unperturbed Blaschke products 

$$
B_a(z)=z^3\frac {z-a}{1-\overline{a}z}
$$

\noindent for $a\in\dis^*$. The maps $B_a$ have $z=0$ and $z=\infty$ as superattracting fixed points of local degree 3.  If $a\in\dis^*$, then they have a single zero, $z_0(a)=a\in\dis^*$, and a single pole $z_{\infty}(a)=1/\overline{a}\in\com\setminus\dis$. Consequently, if $a\in\dis^*$ the unit disk is invariant. Therefore, the basin of attraction of $z=\infty$ is  $A_{a}(\infty)=\wcom\setminus\overline{\dis}$ and the basin of attraction of $z=0$ is  $A_{a}(0)=\dis$. Moreover, the maps $B_a$ have two free critical points $c_+(a)\in\com\setminus\dis$ and $c_-(a)\in\dis^*$ given by

$$
c_{\pm}(a):=a \cdot \frac{1}{3|a|^2}\left(2+|a|^2\pm\sqrt{(|a|^2-4)(|a|^2-1)}\right).
$$

We refer to \cite{CFG1} for a more detailed introduction to the dynamics of the Blaschke products $B_a$. After the singular perturbation, the previously described critical points, zero and pole move continuously with respect to $\lambda$. We obtain from them the critical points $c_\pm(a,\lambda)$, the zero $z_0(a,\lambda)$ and the pole $z_{\infty}(a,\lambda)=1/\overline{a}$ of the maps $B_{a,\lambda}$. The point $z=\infty$ is a permanent superattracting fixed point of $B_{a,\lambda}$ of local degree 3 for all $\lambda$ and, therefore, it is a double critical point. We denote by $A_{a,\lambda}(\infty)$ its basin of attraction and by $A^*_{a,\lambda}(\infty)$ its immediate basin of attraction, i.e.\ the connected component of $A_{a,\lambda}(\infty)$ which contains $z=\infty$.   On the other hand, the point $z=0$ becomes a double preimage of $z=\infty$ after the singular perturbation. Hence, it is a simple critical point. 
 
 So far we have described the position of 5 critical points and one zero of the maps $B_{a,\lambda}$, which have degree $6$. However, every rational map of degree 6 has $2\cdot 6-2=10$ critical points and 6 preimages of $z=0$. The following proposition describes how the remaining 5 critical points and zeros appear in an almost symmetrical position around $z=0$ after the singular perturbation.

\begin{propo}[{\cite[Proposition 2.3]{Can1}}]\label{zeroscrit}
If we fix $a\in\dis^*$, then $B_{a,\lambda}(z)$ has 5 zeros of the form $\xi(\lambda/a)^{1/5}+o(\lambda^{1/5})$, where $\xi$ denotes a fifth root of the unity and $o(\lambda^{1/5})$ is such that $\lim_{\lambda\rightarrow 0} |o(\lambda^{1/5})|/|\lambda^{1/5}|=0$. Moreover, $B_{a,\lambda}(z)$ has  5 critical points of $B_{a,\lambda}(z)$ of the form $-\xi(2\lambda/3a)^{1/5}+o(\lambda^{1/5})$.
\end{propo}

We are particularly interested on the dynamics of the maps $B_{a,\lambda}$ for small parameters $\lambda$. The following theorem describes their dynamics near $z=0$ and in the immediate basin of attraction of $z=\infty$ (see Figure~\ref{esquemacritzeros}).

\begin{thm}[{\cite[Theorem 2.5]{Can1}}]\label{thmcritzeros}

Fix $a\in\dis^*$. Then, there is a constant $\mathcal{C}(a)$ such that if $\lambda\in\dis_{{C}(a)}^*$  the following hold.

\begin{enumerate}[a)]
\item The immediate basin of attraction of $\infty$, $A_{a,\lambda}^*(\infty)$,  is simply connected and $\partial A_{a,\lambda}^*(\infty)$ is a quasicircle. Moreover, $A_{a,\lambda}^*(\infty)$ is mapped with degree $4$ onto itself and contains only a pole $z_{\infty}$ and a critical point $c_+(a,\lambda)$ other than the superattracting fixed point $z=\infty$.

\item There is a simply connected Fatou component $T_0(a,\lambda)$ which contains $z=0$ and is mapped 2 to 1 onto $A_{a,\lambda}^*(\infty)$.

\item There is an annular Fatou component $A_0(a,\lambda)$ which contains 5 critical points and 5 preimages of $z=0$ and is mapped 5 to 1 onto $T_0(a,\lambda)$.

\item The annular region in between $A_0(a,\lambda)$ and $ A_{a,\lambda}^*(\infty)$ contains a critical point $c_-(a,\lambda)$ and a zero $z_0(a,\lambda)$. Moreover, the connected component $D_{0}(a,\lambda)$ of $B_{a,\lambda}^{-1}(T_0(a,\lambda))$ in which  $z_0(a,\lambda)$ lies is simply connected and is mapped with degree 1 onto $T_0(a,\lambda)$. Consequently, it does not contain the critical point $c_-(a,\lambda)$.

\end{enumerate}
\end{thm}

\begin{figure}[hbt!]
\centering
\def\svgwidth{280pt}
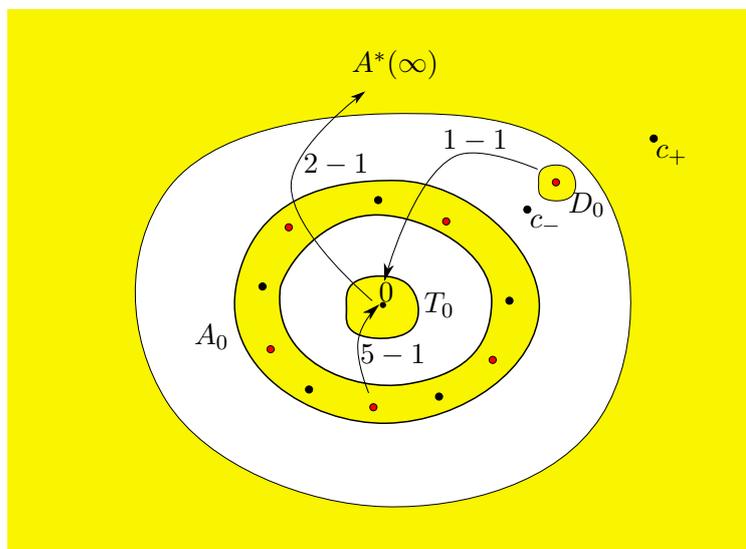
\caption{\small Summary of the dynamics described in Theorem~\ref{thmcritzeros}. We draw in red the preimages of zero and in black the critical points. }
\label{esquemacritzeros}
\end{figure}

Whenever it is clear from the context, we drop the dependence on the parameters $a$ and $\lambda$ of the different sets and points introduced in Theorem~\ref{thmcritzeros}.
It follows from this theorem that if $\lambda\in\dis_{{C}(a)}^*$, where $\mathcal{C}(a)$ depends on $a$, then the dynamics of the maps $B_{a,\lambda}$ is essentially unicritical:  all critical orbits but, maybe, the orbit of the point $c_-$ accumulate on the superattracting fixed point $z=\infty$. Since $z=0$ is the only pole outside $ A^*(\infty)$, we have that if a point $w$ belongs to $A(\infty)\setminus\{A^*(\infty)\cup T_0\}$, then $w$ is eventually mapped under iteration of $B_{a,\lambda}$ into the the annular domain $A_0$ or the disk $D_0$. In Lemma~\ref{fatouinfinit} we show how the connectivity of a Fatou component $U$ depends on whether $U$ is eventually mapped onto $A_0$ or $D_0$. In the proof of Lemma~\ref{fatouinfinit} we use the Riemann-Hurwitz formula (c.f.\ \cite{Mi1, Ste}), which  can be stated as follows. 

\begin{teor}[Riemann-Hurwitz formula]\label{riemannhurwitz}
Let $U$ and $V$ be two connected domains  of $\widehat{\com}$ of finite connectivity $m_{U}$ and $m_{V}$ and let $f:U\rightarrow V$ be a degree $k$ proper map branched over $r$ critical points counted with multiplicity. Then
$$m_{U}-2=k(m_{V}-2)+r.$$
\end{teor}

\begin{lemma}\label{fatouinfinit}
Fix $a\in\dis^*$ and $\lambda\in\dis_{{C}(a)}^*$. Let  $U\subset A(\infty)$ be a Fatou component of $B_{a,\lambda}$ other than $A^*(\infty)$, $T_0$, $A_0$, or $D_0$. Then, exactly one of the following holds.

\begin{itemize}
\item  The set $U$ is simply connected and is eventually mapped under iteration onto $D_0$.

\item  The set $U$ is multiply connected and is eventually mapped under iteration onto $A_0$. Moreover, if $U$ surrounds $z=0$ then $B_{a,\lambda}(U)$ also does.

\item If $c_-$ belongs to a Fatou component which is eventually mapped onto $A_0$, say $U_c$, then $U_c$ is triply connected. 
\end{itemize}
 
\end{lemma}

\proof
The Fatou component $U$ is eventually mapped under iteration onto $D_0$ or $A_0$ since those are the only preimages of $T_0$, and $T_0$ is the only preimage of $A^*(\infty)$ other than itself. It follows from the Riemann-Hurwitz formula (see Theorem~\ref{riemannhurwitz}) that at least two critical points are required to map a multiply connected domain onto a simply connected domain. Since $D_0$ is simply connected and $c_-$ is the only free critical point, we can conclude that all iterated preimages of $D_0$ are simply connected.

Assume that $U$ is eventually mapped onto $A_0$. Since $A_0$ has connectivity $2$, it follows that $U$ has, at least, connectivity $2$ and, hence, it is multiply connected. Indeed, it follows from the Riemann-Hurwitz formula that if a preimage of a doubly connected domain contains no critical point, then it is also doubly connected. Hence, all preimages of $A_0$ are doubly connected until one of them, say $U_c$, contains the critical point $c_-$. It follows again from the Riemann-Hurwitz formula that $U_c$ has connectivity 3. Moreover, it can be proven that if $U_c$ surrounds $z=0$, then it has iterated preimages of arbitrarily large connectivity (see \cite[Proposition 3.1]{Can1}, c.f.\ Theorem~\ref{thmavell}).

 To finish the proof of the lemma it is enough to show that if $U$ does not surround $z=0$, then none of its preimages does. Let $\rm{Bdd}(U)$ denote the set of all points bounded by $U$ (including $U$). Then $\rm{Bdd}(U)$ is a simply connected domain which can contain at most the critical value $B_{a,\lambda}(c_-)$. It follows from the Riemann-Hurwitz formula that all preimages of $\rm{Bdd}(U)$ are simply connected. Moreover, none of them can contain the point $z=0$. Consequently, no preimage of $U$ can surround $z=0$.
\endproof

We finish the preliminaries with the following result. It shows that, if $\lambda\in\dis_{{C}(a)}^*$, then there is a straight annulus inside $A_0$.  Within the paper \cite{Can1} this is presented as one of the conditions to obtain the constant $\mathcal{C}(a)$ in the proof of Theorem 2.5. Even if we could avoid using it in this paper, it helps us understand how the annulus $A_0$ shrinks when $\lambda$ tends to 0.

\begin{lemma}\label{realannulus}
Let $a\in\dis^*$. If $\lambda\in\dis_{{C}(a)}^*$, then $A_0$ contains the straight annulus of inner radius $\left(\frac{|\lambda|}{2|a|}\right)^{1/5}$ and outer radius $\left(\frac{2|\lambda|}{|a|}\right)^{1/5}$.
\end{lemma}

\section{Annular dynamics of the singular perturbations}\label{annulardyn}

The goal of this section is to study the dynamics of the preimages of the annulus $A_0(a,\lambda)$ (see Theorem~\ref{thmcritzeros}), to analyse how they move when we modify the parameter $\lambda$ and to label the preimages of $A_0(a,\lambda)$ which surround $z=0$. To do so we first study how much the singular perturbation changes the dynamics (see Proposition~\ref{conjblas}). Afterwards we study how the maps $B_{a,\lambda}$ act on Jordan curves surrounding $z=0$ (see Proposition\ref{preimcorbes}) and analyse the continuity of $\partial A^*_{a,\lambda}(\infty)$ with respect to $\lambda$ (see Proposition~\ref{contboundary}). Finally, we introduce a labelling of the preimages of $A_0(a,\lambda)$ which is induced by the dynamics of the maps $B_{a,\lambda}$ and allows us to order them depending on their position with respect to $z=0$ (see Definition~\ref{defordering} and Lemma~\ref{ordering}).

 Hereafter we will use the following notation. Given a Jordan curve $\eta\subset \com$, we denote by $\rm{Int}(\eta)$ the bounded component of $\com\setminus\eta$  and by $\rm{Ext}(\eta)$ the unbounded component of $\com\setminus\eta$. 
 The following proposition tells us that the singular perturbation only modifies significantly the dynamics in the region bounded by the annulus $A_0(a,\lambda)$. 

\begin{propo}\label{conjblas}
Let $a\in \dis^*$ and let $\lambda\in\dis_{{C}(a)}^*$. Then, there exists an analytic Jordan curve $\Gamma\subset A_0(a,\lambda)$ such that $B_{a,\lambda}$ is conjugate to a Blaschke product $B_{b,t}(z)=e^{2\pi i t} z^3(z-b)/(z-\overline{b})$, where $b\in \dis^*$ and $|e^{2\pi i t}|=1$,  in the annulus bounded by $\Gamma$ and $\partial A_{a,\lambda}^*(\infty)$.
\end{propo}

\proof

The main idea of the proof is to perform a quasiconformal surgery which erases the extra dynamics which appear after the singular perturbation near $z=0$ and keeps the dynamics of the Blaschke products within the unbounded region delimited by an analytic curve $\Gamma\subset A_0$. We refer to \cite{Ah} and \cite{BF} for an introduction to quasiconformal mappings and the tools used in quasiconformal surgery.

 We first show how to obtain the curve $\Gamma$. Let $\gamma$ be an analytic Jordan curve in the Fatou component $T_0$ (see Theorem~\ref{thmcritzeros}) such that it surrounds $z=0$ and the images of all critical points contained in $A_0$. Then, the annulus $\mathcal{A}$ bounded by $\gamma$ and $\partial T_0$ contains no critical values and all components of $B_{a,\lambda}^{-1}(\mathcal{A})$ are annuli bounded by preimages of $\gamma$ and $\partial T_0$. Since $\partial B_{a,\lambda}^{-1}(T_0)$ consists of three connected components (notice that $\partial B_{a,\lambda}^{-1}(T_0)=\partial D_0 \cup \partial A_0$), we conclude that $\partial B_{a,\lambda}^{-1}(\mathcal{A})$ consists of three annular connected components. Two of them are in $A_0$. Since $\mathcal{A}$ does neither contain $z=0$ nor any critical value, one of the connected components of $\partial B_{a,\lambda}^{-1}(\mathcal{A})$, say $\mathcal{A}_o^{-1}$, surrounds all critical points and preimages of $z=0$ in $A_0$. The other component of $\partial B_{a,\lambda}^{-1}(\mathcal{A})$ in $A_0$, say $\mathcal{A}_i^{-1}$, does not surround any zero or critical point in $A_0$. We denote by $\gamma_i^{-1}$ and $\gamma_o^{-1}$ the connected components of $\partial\mathcal{A}_{i}^{-1}$ and $\partial\mathcal{A}_{o}^{-1}$ contained in $A_0$ (see Figure~\ref{esquemacorbes}). The curves $\gamma_i^{-1}$ and $\gamma_o^{-1}$ are mapped under $B_{a,\lambda}$ onto $\gamma$. Since $\gamma$ is an analytic Jordan curve which contains no critical value, the curves $\gamma_{i}^{-1}$ and $\gamma_{o}^{-1}$ are also analytic Jordan curves. 
 
 \begin{figure}[hbt!]
\centering
\def\svgwidth{280pt}
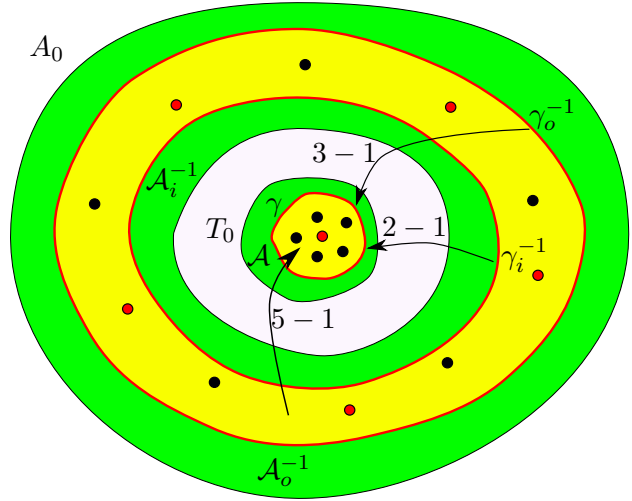
\caption{\small Scheme of the dynamics described in the proof of Proposition \ref{conjblas}. We draw as red dots $z=0$ and its preimages in $A_0$. We draw as black dots the critical points in $A_0$ and its images in $T_0$. }
\label{esquemacorbes}
\end{figure}

 Now we want to show that $\gamma_o^{-1}$ is mapped 3 to 1 onto $\gamma$. Since $A_0$ is mapped 5 to 1 onto $T_0$  (see Theorem~\ref{thmcritzeros}), it is enough to show that $\gamma_i^{-1}$ is mapped 2 to 1 onto $\gamma$. Consider the annulus $\mathcal{A}'$ bounded by $\partial T_0$ and $\gamma_i^{-1}$. It is mapped under $B_{a,\lambda}$ onto the annulus $\mathcal{A}''$ bounded by $\partial A^*(\infty)$ and $\gamma$. It is not difficult to see that $\mathcal{A}'$ is a connected component of $B_{a,\lambda}^{-1}(\mathcal{A}'')$ and, hence, $B_{a,\lambda}|_{\mathcal{A}'}$ is proper. In particular, $B_{a,\lambda}|_{\mathcal{A}'}$ has a degree $k$ which is also achieved on the boundaries. Since $T_0$ is mapped onto $A^*(\infty)$ 2 to 1 (see Theorem~\ref{thmcritzeros}), we conclude that this degree is 2 and $\gamma_{i}^{-1}$ is mapped 2 to 1 onto $\gamma$ .

 Let $\Gamma:=\gamma_0^{-1}$.  We have proven that $\Gamma$ is an analytic Jordan curve which is mapped 3 to 1 onto the analytic Jordan curve $\gamma$. Moreover, $\gamma\subset \rm{Int}(\Gamma)$. It is an standard procedure of quasiconformal surgery used in complex dynamics to `glue' $z^3$ in $\rm{Int}(\Gamma)$ under these circumstances (c.f.\ \cite[proof of Theorem 7.4]{BF}). We include the details of the surgery construction in sake of completeness.

Pick $0<\rho <1$. Let $\mathcal{R}:\rm{Int}(\gamma)\rightarrow\dis_{\rho^3}$ be a Riemann map fixing $z=0$. Then, $\mathcal{R}$ extends to the boundary as an analytic map. Indeed, it extends continuously to the boundary by Carath\'eodory's Theorem. Since $\gamma$ is an analytic curve, this continuation is analytic \linebreak (see \cite[Thm.\ 2.9]{BF}).  Let  $\psi_1:\gamma\rightarrow\cercle_{\rho^3}$  be the extension map, where $\cercle_{\rho^3}:=\partial\dis_{\rho^3}$. Since $\psi_1\circ B_{a,\lambda}:\Gamma\rightarrow \cercle_{\rho^3}$ is an analytic map of degree $3$, we can choose a $C^1$ (or real analytic) lift $\psi_2:\Gamma \rightarrow \cercle_{\rho}$ so that $\psi_1(B_{a,\lambda}(z))=\psi_2(z)^3$.

Let $A$ be the annulus bounded by $\gamma$ and $\Gamma$ and define $\mathbb{A}_{\rho, \rho^3}=\{z; \, \rho<|z|<\rho^3\}$. Then, there exists a quasiconformal map $\psi: A\rightarrow \mathbb{A}_{\rho, \rho^3}$ which extends continuously to $\psi: \overline{A}\rightarrow \overline{\mathbb{A}_{\rho, \rho^3}}$ and such that $\psi|_{\Gamma=}\psi_2$ and $\psi|_{\gamma}=\psi_1$ (see e.g.\ \cite{Ah} or \cite{BF}).  Now we define a model map $F$ as

$$F(z)=\left\{\begin{array}{lcl}
B_{a,\lambda}(z) &  \mbox{ for } & z\in \rm{Ext}(\Gamma)\\
\mathcal{R}^{-1}\left(\psi(z)^3\right)&  \mbox{ for } & z\in \rm{Ext}(\gamma)\setminus \rm{Ext}(\Gamma)\\
\mathcal{R}^{-1}\left(\mathcal{R}(z)^3\right)&  \mbox{ for } & z\in\overline{\rm{Int}(\gamma)}.
\end{array}\right. $$

The map $F(z)$ is quasiregular and has topological degree 4 by construction. Indeed, since $A_0$ is mapped with degree 5 onto $T_0$,  $\rm{Int}(\Gamma)$ contains $5$ preimages of every point of $\rm{Int}(\gamma)$. Moreover, since $B_{a,\lambda}|_{\rm{Int}(\gamma_i^{-1})}:\rm{Int}(\gamma_i^{-1})\rightarrow \rm{Ext}(\gamma)$ is proper of degree 2 (notice that \linebreak $T_0\subset \rm{Int}(\gamma_i^{-1})$ and $B_{a,\lambda}|_{T_0}$ has degree 2), $\rm{Int}(\Gamma)$ contains 2 preimages of every point in $\rm{Ext}(\gamma)$ under $B_{a,\lambda}$. Instead, $F$ maps $\rm{Int}(\Gamma)$ onto $\rm{Int}(\gamma)$ with degree 3. Hence, the surgery construction erases two preimages of every point in $\wcom$ and we can conclude that $F$ has topological degree $4$ since $B_{a,\lambda}$ has degree 6. We continue by defining an $F$-invariant Beltrami coefficient $\mu$. Observe that $F$ is a holomorphic map everywhere except on the annulus $A$ bounded by $\gamma$ and $\Gamma$. Notice also that the orbit of a point $z$ can go at most once through $A$. Let $A^n=\{z\;|\; B_{a,\lambda}^n(z)\in A\}$. Thus, if $\mu_0\equiv 0$ denotes the Beltrami coefficient of the standard complex structure, we can define

$$\mu(z)=\left\{\begin{array}{lcl}
\psi^*\mu_0(z) &  \mbox{ for } & z\in A\\
(B_{a,\lambda})^{*}\mu(z) &  \mbox{ for } & z\in A^n\setminus A^{n-1}\\
\mu_0(z)&  & elsewhere,
\end{array}\right. $$

\noindent
where $^*$ denotes the pullback operation. By construction, $F^{*}\mu=\mu$. Moreover, since $B_{a,\lambda}$ is a holomorphic function and hence preserves dilatation, we have that the dilatation of $\mu$ is given by $K_{\mu}=K_{\psi^*\mu_0 }$. Thus, $\mu$ has bounded dilatation. Let $\phi$ be an integrating map given by the Measurable Riemann Mapping Theorem (see \cite{AB}) fixing $z=0$ and $z=\infty$. Then, $\phi^*\mu_0=\mu$. Finally, define $f=\phi \circ F\circ \phi^{-1}$. By construction, $f^* \mu_0=\mu_0$.  Then, by Weyl's Lemma (see \cite[p. 16]{Ah}), $f$ is a holomorphic map. Moreover, since $F$ has topological degree $4$, we conclude that $f$ is a rational map of degree 4.

Summarizing, with the quasiconformal surgery procedure we have obtained a rational map $f$ such that the following hold.

\begin{itemize}
\item The maps $B_{a,\lambda}$ and $f$ are conjugate on $\rm{Ext}(\Gamma)$.
\item The map $f$ is conjugate to $z^3$ on $\rm{Int}(\phi(\gamma))$. In particular, it has $z=0$ as superattracting fixed point of local degree 3.
\item The only pole of $f$ is given by $\phi(z_{\infty})$, where $z_{\infty}=1/\overline{a}$ is the only pole of $B_{a,\lambda}$ in $\rm{Ext}(\Gamma)$. Moreover, $\phi(z_{\infty})$ belongs to the basin of attraction of $z=\infty$ under $f$.

\end{itemize}

Let $\mathcal{D}=\phi(\com\setminus \overline{A^*(\infty)})$. Then, since $f$ has no pole in $\mathcal{D}$ we conclude that $f$ maps $\mathcal{D}$ onto itself with degree 4 (it is a fully invariant set under $f$). 
Consider the  Riemann map $\Phi$ sending $\mathcal{D}$ onto $\dis$ and $0$ onto $0$. Then $f|_{\mathcal{D}}$ is conjugate to the map $\tilde{f}:\dis\rightarrow\dis$ defined as $\tilde{f}=\Phi\circ f\circ \Phi^{-1}$. Since $f|_{D}$ has degree 4 and $z=0$ is a superattracting fixed point of local degree 3, it follows that $\tilde{f}$ is necessarily a Blaschke product of the form $B_{b,t}(z)=e^{2\pi i t} z^3(z-b)/(z-\overline{b})$, where $b\in\dis\setminus\{0\}$ and $|e^{2\pi i t}|=1$ (see \cite[Lemma 15.5]{Mi1}). Finally, the map $\Phi\circ\phi$ conjugates $B_{a,\lambda}$ with the Blaschke product $B_{b,t}$ on the annulus bounded by $\Gamma$ and $\partial A^*(\infty)$.

\endproof

Given a Jordan curve $\gamma$ which surrounds $z=0$, we want to analyse the structure of $B_{a,\lambda}^{-1}(\gamma)$. This is done in Proposition~\ref{preimcorbes}, which uses the previous result. Before stating this proposition we introduce notation which is used in the remaining parts of this section.

\begin{defin}\label{defininout}
We define $\mathcal{A}_{0,\infty}(a,\lambda)$ to be the annulus bounded by $\partial T_0(a,\lambda)$ and $\partial A_{a,\lambda}^*(\infty)$, $\mathcal{A}_{in}(a,\lambda)$ to be the annulus bounded by $\partial T_0$ and $\overline{A_0(a,\lambda)}$, and $\mathcal{A}_{out}(a,\lambda)$ to be the annulus bounded by $\overline{A_0(a,\lambda)}$ and $\partial A_{a,\lambda}^*(\infty)$.

\end{defin}

As before, we should drop the dependence on $a$ and $\lambda$ of the annulus introduced in Definition~\ref{defininout} whenever it is clear from the context. In Figure~\ref{esquemacritzeros}, $\mathcal{A}_{in}$ corresponds to the interior white annulus, while $\mathcal{A}_{out}$ corresponds to the union of the exterior white region together with the yellow disk $D_0$.

\begin{propo}\label{preimcorbes}
Let $\gamma$ be a Jordan curve which surrounds $z=0$ and is contained in the annulus $\mathcal{A}_{0,\infty}$. Then $B_{a,\lambda}^{-1}(\gamma)$ has a connected component contained in $\mathcal{A}_{in}$ which is a Jordan curve that surrounds $z=0$ and  is mapped 2-1 onto $\gamma$. The other connected components of $B_{a,\lambda}^{-1}(\gamma)$ belong to $\mathcal{A}_{out}$ and exactly one of the following holds.

\begin{itemize}
\item If $B_{a,\lambda}(c_-)\in\rm{Int}(\gamma)$,  then $B_{a,\lambda}^{-1}(\gamma)$ has a single connected component in $\mathcal{A}_{out}$, which is a Jordan curve that surrounds $z=0$ and is mapped 4-1 onto $\gamma$

\item If $B_{a,\lambda}(c_-)\in\gamma$,  then $B_{a,\lambda}^{-1}(\gamma)$ has a single connected component in $\mathcal{A}_{out}$, which consists of the union of 2 Jordan curves which intersect at the critical point $c_-$. One of them surrounds $z=0$ and is mapped 3-1 onto $\gamma$. The other one surrounds the zero $z_0$ but not $z=0$ and is mapped 1-1 onto $\gamma$.

\item If $B_{a,\lambda}(c_-)\in\rm{Ext}(\gamma)$,  then $B_{a,\lambda}^{-1}(\gamma)$ has two disjoint connected components in $\mathcal{A}_{out}$, which are Jordan curves. One of them surrounds $z=0$ and is mapped 3-1 onto $\gamma$. The other one surrounds the zero $z_0$ but not $z=0$ and is mapped 1-1 onto $\gamma$.
\end{itemize}

\end{propo}

\proof
Given that $B_{a,\lambda}(A_0)=T_0$, we have $B_{a,\lambda}^{-1}(\mathcal{A}_{0,\infty})\subset \mathcal{A}_{in}\cup \mathcal{A}_{out}$. We first analyse the set $B_{a,\lambda}^{-1}(\gamma)$ in $\mathcal{A}_{in}$. Since there is no zero and no pole in $\mathcal{A}_{in}$, we have that $\mathcal{A}_{in}$ is mapped onto $\mathcal{A}_{0,\infty}$ under $B_{a,\lambda}$. Moreover, $B_{a,\lambda}|_{\mathcal{A}_{in}}$ is proper of degree 2 since $B_{a,\lambda}|_{\partial T_0}$ has degree 2. A connected component $\gamma'$ of $B_{a,\lambda}^{-1}(\gamma)$ in $\mathcal{A}_{in}$ surrounds $z=0$. Otherwise, since there is no zero and no pole in $\mathcal{A}_{in}$, it could not be mapped to a curve surrounding $z=0$. As before, $B_{a,\lambda}$ is proper of degree 2 on the annulus bounded by $\partial T_0$ and $\gamma'$. In particular, $\gamma'$ is mapped 2 to 1 onto $\gamma$  and there is no other connected component of $B_{a,\lambda}^{-1}(\gamma)$  in $\mathcal{A}_{in}$.

The set  $\mathcal{A}_{out}$ belongs to the annular region where the dynamics of $B_{a,\lambda}$ is conjugate to the dynamics of a Blaschke product of the form $B_{b,t}(z)=e^{2\pi i t} z^3(z-b)/(z-\overline{b})$ as in Proposition~\ref{conjblas}. Hence, to finish the proof of the result it is enough to prove that the last three statements of the proposition  hold for $B_{b,t}$ in $\dis\setminus\{0\}$. 

The Blaschke products $B_{b,t}$ map the unit disk $\dis$ into itself. They have $z=0$ as superattracting fixed point of local degree 3, which has $\dis$ as its basin of attraction. Moreover, there is a unique critical point $c_{b,t}\in\dis$ and a unique zero $w_{b,t}\in\dis$. Notice that the map $B_{b,t}|_{\dis}$ has degree 4.
Let $\eta\subset\dis$ be a Jordan curve surrounding $z=0$. Since there is only a critical point, any connected component $\eta'$ of $B_{b,t}^{-1}(\eta)$ is either a Jordan curve or the union of two Jordan curves which intersect at the critical point. Since $B_{b,t}|_{\dis}$ has no poles, any bounded connected component of $\com\setminus\eta'$ is mapped into $\rm{Int}(\eta)$. Moreover, exactly one of the three following cases holds.

 If $B_{b,t}(c_{b,t})\in\rm{Ext}(\eta)$ then there is a connected  component $\eta'$ of $B_{b,t}^{-1}(\eta)$ which is a Jordan curve that surrounds $z=0$ but not the critical point. Then $\rm{Int}(\eta')$ is mapped onto $\rm{Int}(\eta)$ with degree 3. Consequently, $\eta'$ is mapped onto $\eta$ with degree 3. There is an extra preimage $\eta''$ which is a Jordan curve that does not surround $z=0$, is mapped 1 to 1 onto $\eta$ and such that $w_{b,t}\in \rm{Int}(\eta'')$.

If $B_{b,t}(c_{b,t})\in\eta$ then $B_{b,t}^{-1}(\eta)$ is the union of two Jordan curves $\eta'$ and $\eta''$. As in the previous case we have that $0\in\rm{Int}(\eta')$, that $w_{b,t}\in\rm{Int}(\eta'')$, and that $\eta'$ and $\eta''$ are mapped into $\eta$ with degree 3 and 1, respectively.

 If $B_{b,t}(c_{b,t})\in \rm{Int}(\eta)$ then there is a component of $B_{b,t}^{-1}(\eta)$ which surrounds both $z=0$ and the critical point $c_{b,t}$. If not, there would be two connected components of $B_{b,t}^{-1}(\eta)$, one surrounding $z=0$ and the other surrounding the critical point. The interior of these connected components would be mapped into $\rm{Int}(\eta)$ with degree 3 and 2, respectively, which is not possible since $B_{b,t}|_{\dis}$ has degree 4. Let $\eta'$ be the connected component of $B_{b,t}^{-1}(\eta)$ surrounding $z=0$ and the critical point. Then $B_{b,t}|_{\rm{Int}(\eta')}$ is a proper map from a simply connected domain onto a simply connected domain with 3 critical points counted with multiplicity (recall that $z=0$ is a double critical point). Then, it follows from the Riemann-Hurwitz formula (see Theorem~\ref{riemannhurwitz}) that $B_{b,t}|_{\rm{Int}(\eta')}$ has degree 4. Hence, $\eta'$ is mapped onto $\eta$ with degree 4 and there is no other connected component of $B_{b,t}^{-1}(\eta)$.

\endproof

The following result is a direct corollary of Proposition~\ref{conjblas}. It describes how preimages of $A_0(a,\lambda)$ accumulate on $\partial A_{a,\lambda}^*(\infty)$. It also provides a first ordering in a subset of the Fatou components which are eventually mapped onto  $A_0(a,\lambda)$.

\begin{co}\label{anellsacumulen}
Let $a\in \dis^*$ and let $\lambda\in\dis_{\mathcal{C}(a)}^*$. Then, there is a sequence  $\{A_{n}(a,\lambda)\}_{n\in\nat}$ of connected components of $B_{a,\lambda}^{-n}(A_0(a,\lambda))$ such that the following hold.

\begin{itemize}
\item Every $A_n(a,\lambda)$  surrounds $z=0$. 
\item For every $n\in\nat$, $B_{a,\lambda}(A_n(a,\lambda))=A_{n-1}(a,\lambda)$ and $A_{n-1}$ is a subset of the bounded connected component of $\wcom\setminus A_n$ which contains $z=0$. 
\item The sets $A_n(a,\lambda)$ accumulate on $\partial A_{a,\lambda}^*(\infty)$.
\end{itemize}
  
\end{co}
\proof

 Let $B_{b,t}(z)=e^{2\pi i t} z^3(z-b)/(z-\overline{b})$ where $b\in\dis^*$ and $|e^{2\pi i t}|=1$ as in Proposition~\ref{conjblas}. Then, the preimage of any Jordan curve $\gamma\subset \dis\setminus \{0\}$ surrounding 0 under $B_{b,t}$  has one connected component, say $\gamma_{-1}$, which also surrounds $z=0$. If $\gamma\cap\gamma_{-1}=\emptyset$ then $\gamma\subset \rm{Int}(\gamma_{-1})$. Moreover, the iterated preimages of this curve  accumulate on $\cercle$, which is the boundary of the basin of attraction of $z=0$. The claims hold by noticing that the exterior boundary of $A_0(a,\lambda)$ belongs to the annulus bounded by $\Gamma$ and $\partial A^*(\infty)$ where the dynamics of $B_{a,\lambda}$ is conjugate to  a Blaschke product $B_{b,t}(z)$ (see Proposition~\ref{conjblas}). 
\endproof

In Corollary~\ref{anellsacumulen} we have introduced  a sequence of preimages $A_n$ of $A_0$ with prescribed dynamics. However, it is convenient to label all the preimages of $A_0$ which surround $z=0$. We do this in the next lemma.

\begin{lemma}\label{deforderwell}
Given any finite sequence $\Delta=\{{i_0,\cdots,i_{p_{\Delta}-1}}\}$, where $p_{\Delta}$ is a non-zero natural number and $i_n\in\{0,1\}$ for all integer $0\leq n <p_{\Delta}$, there exists a unique Fatou component $A_{\Delta}(a,\lambda)$ which surrounds $z=0$, is mapped onto $A_0(a,\lambda)$ in $p_{\Delta}$ iterations of $B_{a,\lambda}$ and satisfies the following.

\begin{itemize}
\item If $i_n=0$, then $B_{a,\lambda}^{n}(A_{\Delta}(a,\lambda))\subset \mathcal{A}_{out}$. In particular, if $i_0=0$ then  $A_{\Delta}(a,\lambda) \subset \mathcal{A}_{out}$.

\item If $i_n=1$, then $B_{a,\lambda}^{n}(A_{\Delta}(a,\lambda))\subset \mathcal{A}_{in}$. In particular, if $i_0=1$ then  $A_{\Delta}(a,\lambda) \subset \mathcal{A}_{in}$.
\end{itemize}
\end{lemma}

\proof
 The result follows directly from Lemma~\ref{fatouinfinit} and Proposition~\ref{preimcorbes}. From Lemma~\ref{fatouinfinit} we know that if $U$ is a Fatou component that surrounds $z=0$ then $B_{a,\lambda}(U)$ also surrounds $z=0$. Therefore, it is enough to analyse the preimages of any Fatou component surrounding $z=0$. From Proposition~\ref{preimcorbes} we know that any Fatou component $U$ which surrounds $z=0$ has exactly two preimages which surround $z=0$, one in $\mathcal{A}_{in}$ and the other in $\mathcal{A}_{out}$.
\endproof

\begin{rem}\label{deforderconflict}
Along the paper we keep the notation $A_n(a,\lambda)$ for the Fatou components introduced in Corollary~\ref{anellsacumulen}. The Fatou component $A_n(a,\lambda)$ corresponds to $A_{\Delta}(a,\lambda)$ where $\Delta$ consists of a sequence of $n$ symbols 0. Since we keep the notation $A_n$, we avoid the confusion between $A_1(a,\lambda)=A_{\{0\}}(a,\lambda)\subset\mathcal{A}_{out}$ and $A_{\{1\}}(a,\lambda)\subset\mathcal{A}_{in}$, or the original annulus $A_0(a,\lambda)$.

\end{rem}

Notice that if $c_-(a,\lambda)\in U_c(a,\lambda)$, where $U_c(a,\lambda)$ is a preimage of $A_0(a,\lambda)$, then $U_c(a,\lambda)$ has connectivity $3$ (see Lemma~\ref{fatouinfinit}). If $U_c(a,\lambda)$ surrounds $z=0$, the subsequent preimages of $U_c(a,\lambda)$ which also surround $z=0$  have greater connectivity (see \cite{Can1}). Because of this, when we talk about the sets $A_{\Delta}(a,\lambda)$ we  refer to them as multiply connected preimages of $A_0$ surrounding $z=0$.

The goal of the rest of this section is to study how the  Fatou components  $A_{\Delta}(a,\lambda)$ depend on the parameters. It  will be convenient to denote them keeping their dependence with respect to $a$ and $\lambda$.  Since their boundaries are eventually mapped under iteration onto $\partial A^*_{a,\lambda}(\infty)$, we first analyse how this set depends on $\lambda$.

\begin{propo}\label{contboundary}
Let $a\in\dis^*$. Then, for all $\lambda\in\dis_{\mathcal{C}(a)}$, the set  $\partial A^*_{a,\lambda}(\infty)$ depends continuously on $\lambda$.

\end{propo}

\proof

The idea of this proof is to build a holomorphic family of polynomial-like mappings for which the dynamics on $A_{a,\lambda}^*(\infty)$ is preserved. A polynomial-like map is a triple $\{f,U,V\}$ where $U$ and $V$ are simply connected domains bounded by analytic curves, $U$ is compactly contained in $V$, and $f:U\rightarrow V$ is holomorphic and proper of a given degree $k$  (see~\cite{DH1}). The filled Julia set $\mathcal{K}_f$ of a polynomial-like map is defined as the set of points which never escape $U$ under iteration of $f$, i.e.\ $\mathcal{K}_f=\{z\in U \; | \; f^n(z)\in U \; \forall n\in\nat\}$. Its Julia set is defined as $\mathcal{J}_f=\partial\mathcal{K}_f$. 

We begin by building a family of polynomial-like mappings $\{f_{\lambda},U_{\lambda},V\}$ depending holomorphically on $\lambda$ and such that $\mathcal{J}_{f_{\lambda}}=\partial A^*_{a,\lambda}(\infty)$.
Fix $\lambda_0\in\dis_{\mathcal{C}(a)}^*$. By Corollary~\ref{anellsacumulen}, we know that there are multiply connected Fatou components  $A_n(a,\lambda_0)$ which  accumulate on $\partial A_{a,\lambda_0}^*(\infty)$ and surround $z=0$. Hence, we can fix $n$ such that $A_n(a,\lambda_0)$  surrounds $z=0$ and the critical value $v_-(a,\lambda_0)=B_{a,\lambda_0}(c_-(a,\lambda_0))$. We can take an analytic Jordan curve $\gamma\subset A_n(a,\lambda_0)$ surrounding $z=0$ and the critical value $v_-(a,\lambda_0)$. By proposition~\ref{preimcorbes} we know that there is a preimage $\gamma^{-1}_{\lambda_0}\subset A_{n+1}(a,\lambda_0)$ of $\gamma$ which surrounds $z=0$ and is mapped with degree 4 onto $\gamma$. The curve $\gamma^{-1}_{\lambda_0}$ also surrounds the critical point $c_-(a,\lambda_0)$.

 The triple $\{f_{\lambda_0}=B_{a,\lambda_0}, \rm{Ext}(\gamma^{-1}_{\lambda_0}), \rm{Ext}(\gamma)\}$ is a polynomial-like map of degree 4. Moreover, we have that $\mathcal{K}_{f_{\lambda_0}}=\overline{A_{a,\lambda_0}^*(\infty)}$ and that $\mathcal{J}_{f_{\lambda_0}}=\partial A^*_{a,\lambda_0}(\infty)$. If we fix $\gamma$ and take $\lambda$ in a neighbourhood $\Lambda$ of $\lambda_0$, then there is a connected component $\gamma^{-1}_{\lambda}$ of $B_{a,\lambda}^{-1}(\gamma)$ which surrounds $z=0$, moves holomorphically with respect to $\lambda$ and coincides with $\gamma_{\lambda_0}^{-1}$ for $\lambda=\lambda_0$. If the neighbourhood of $\Lambda$ is small enough, then $\gamma^{-1}_{\lambda}$ is an analytic Jordan curve,   $\gamma\subset \rm{Int}(\gamma^{-1}_{\lambda})$, $\gamma^{-1}_{\lambda}$ surrounds the critical point $c_-(a,\lambda)$ and $\gamma^{-1}_{\lambda}$ is mapped onto $\gamma$ with degree $4$ under $B_{a,\lambda}$. Hence, we have a family of polynomial-like maps $\{f_\lambda=B_{a,\lambda}, \rm{Ext}(\gamma^{-1}_{\lambda}), \rm{Ext}(\gamma)\}_{\lambda\in\Lambda}$ which depend holomorphically on $\lambda$ and such that $\mathcal{J}_{f_{\lambda}}=\partial A^*_{a,\lambda}(\infty)$. Moreover, since the only critical points of $B_{a,\lambda}$ in $\rm{Ext}(\gamma^{-1}_{\lambda})$ are $c_+(a,\lambda)$ and $z=\infty$ and both of them belong to $A^*_{a,\lambda}(\infty)$, it follows that the family of polynomial-like maps is $\mathcal{J}$-stable and the sets $\mathcal{J}_{f_{\lambda}}=\partial A^*_{a,\lambda}(\infty)$ depend continuously on $\lambda$ (see \cite[Proposition~10]{DH1}).

To finish the proof we need to show that the same construction can be made in a neighbourhood of the parameter $\lambda=0$. To do so, it is enough to find  analytic Jordan curves $\gamma,\gamma_0^{-1}\subset \dis$ such that $\gamma$ surrounds $z=0$ and the critical value $v_-(a)$, that $\gamma_0^{-1}$ surrounds $z=0$ and the critical point $c_-(a)$ and is mapped with degree 4 onto $\gamma$ under $B_{a,0}=B_a$, and that $\gamma\subset\rm{Int}(\gamma_0^{-1})$. The existence of the curves $\gamma$ and $\gamma_0^{-1}$ for $B_a$ is shown in \cite[Proposition~2.2]{Can1}. Using the curve $\gamma$ we can build a family of $\mathcal{J}$-stable polynomial-like maps $\{f_\lambda=B_{a,\lambda}, \rm{Ext}(\gamma^{-1}_{\lambda}), \rm{Ext}(\gamma)\}_{\lambda\in\Lambda}$. As before, we conclude that the sets $\mathcal{J}_{f_{\lambda}}=\partial A^*_{a,\lambda}(\infty)$ depend continuously on $\lambda$  for all $\lambda\in\Lambda$.

\endproof

\begin{rem} It follows from Proposition~\ref{contboundary} that the quasicircles $\partial A^*_{a,\lambda}(\infty)$ are continuous deformations of the unit circle  $\cercle=\partial A^*_{a,0}(\infty)$.
\end{rem}

The next result is a direct corollary of Proposition~\ref{contboundary}. The proof is straightforward using that $\partial A_0(a,\lambda)\subset B_{a,\lambda}^{-2}(\partial A^*_{a,\lambda}(\infty))$ and that there is no critical point on $\partial T_0(a,\lambda) \cup\partial A_0(a,\lambda)$.  

\begin{co}\label{continuitatanells}
The annulus $A_0(a,\lambda)$ moves continuously with respect to $\lambda$ for all $\lambda\in\dis_{\mathcal{C}(a)}^*$. Moreover, fixed any $n\in\nat$, the set $B_{a,\lambda}^{-n}(A_0(a,\lambda))$ moves continuously with respect to $\lambda$.

\end{co}

The previous corollary applies to the union of all Fatou components which are mapped under $n$ iterations onto $A_0(a,\lambda)$. However, we have to be more careful when talking about the continuity of a concrete preimage $A(a,\lambda)$ of $A_0(a,\lambda)$ with respect to $\lambda$. It follows from the Riemann-Hurwitz formula (see Theorem~\ref{riemannhurwitz})  that all preimages of the annulus $A_0(a,\lambda)$ are doubly connected until one of them, say $U_c(a,\lambda)$, contains the critical point $c_-(a,\lambda)$ (compare with the proof of Lemma~\ref{fatouinfinit}). In Lemma~\ref{fatouinfinit} we  show that, if such $U_c(a,\lambda)$ exists, it has connectivity 3. 
The following proposition describes how is the transition from double to triple connectivity  when $c_-(a,\lambda)\in\partial U_c(a,\lambda)$ (see Figure \ref{esquemapunxat}). 

\begin{propo}\label{confpinching}
Assume that there is a preimage $U_c(a,\lambda)$ of the annulus $A_0(a,\lambda)$ such that $c_-(a,\lambda)\in\partial U_c(a,\lambda)$. Then $U_c(a,\lambda)$ is doubly connected and exactly one of the following holds.

\begin{itemize}

\item One connected component of $\partial U_c(a,\lambda)$ is a Jordan curve. The other connected component of $\partial U_c(a,\lambda)$ consists of the union of two Jordan curves attached at the critical point $c_-(a,\lambda)$.

\item Both connected components of $\partial U_c(a,\lambda)$ are Jordan curves. There is an extra preimage $U_c'(a,\lambda)$ of $A_0(a,\lambda)$ such that $\partial U_c(a,\lambda)\cap \partial U_c'(a,\lambda)=c_-(a,\lambda)$.  Moreover, if $U_c(a,\lambda)$ surrounds $z=0$, then $U_c'(a,\lambda)$ cannot surround $z=0$.

\end{itemize}

\end{propo}

\begin{proof}
 Since $c_-(a,\lambda)$ is not mapped under iteration into $A_0(a,\lambda)$, it follows from the Riemann-Hurwitz formula (see Theorem~\ref{riemannhurwitz}) that all preimages of $A_0(a,\lambda)$ are doubly connected, and so is $U_c(a,\lambda)$. Moreover, all connected components of $\partial B_{a,\lambda}^{-n}(A_0(a,\lambda))$ are Jordan curves until we find some $n>0$ such that $B_{a,\lambda}^n(c_-(a,\lambda))\in\partial A_0(a,\lambda)$. In particular, $\partial B_{a,\lambda}(U_c(a,\lambda))$ is the union of two Jordan curves, say $\beta_1$ and $\beta_2$. We may assume, without lost of generality, that $B_{a,\lambda}(c_-(a,\lambda))\in\beta_1$. Let $\beta_1'$ and $\beta_2'$  be the  connected components of the preimages of $\beta_1$ and $\beta_2$, respectively, which have no-empty intersection with $\partial U_c(a,\lambda)$. Then $\beta_1'$ consists of the union of two Jordan curves which intersect at $c_-(a,\lambda)$ while $\beta_2'$ is a Jordan curve. By Proposition~\ref{preimcorbes}, at most one of the two Jordan curves which form $\beta_1'$ can surround $z=0$. We  conclude that only the two possibilities of the statement can take place. 

\end{proof}

\begin{figure}[hbt!]
    \centering
    
     \subfigure{
    \def\svgwidth{100pt}
    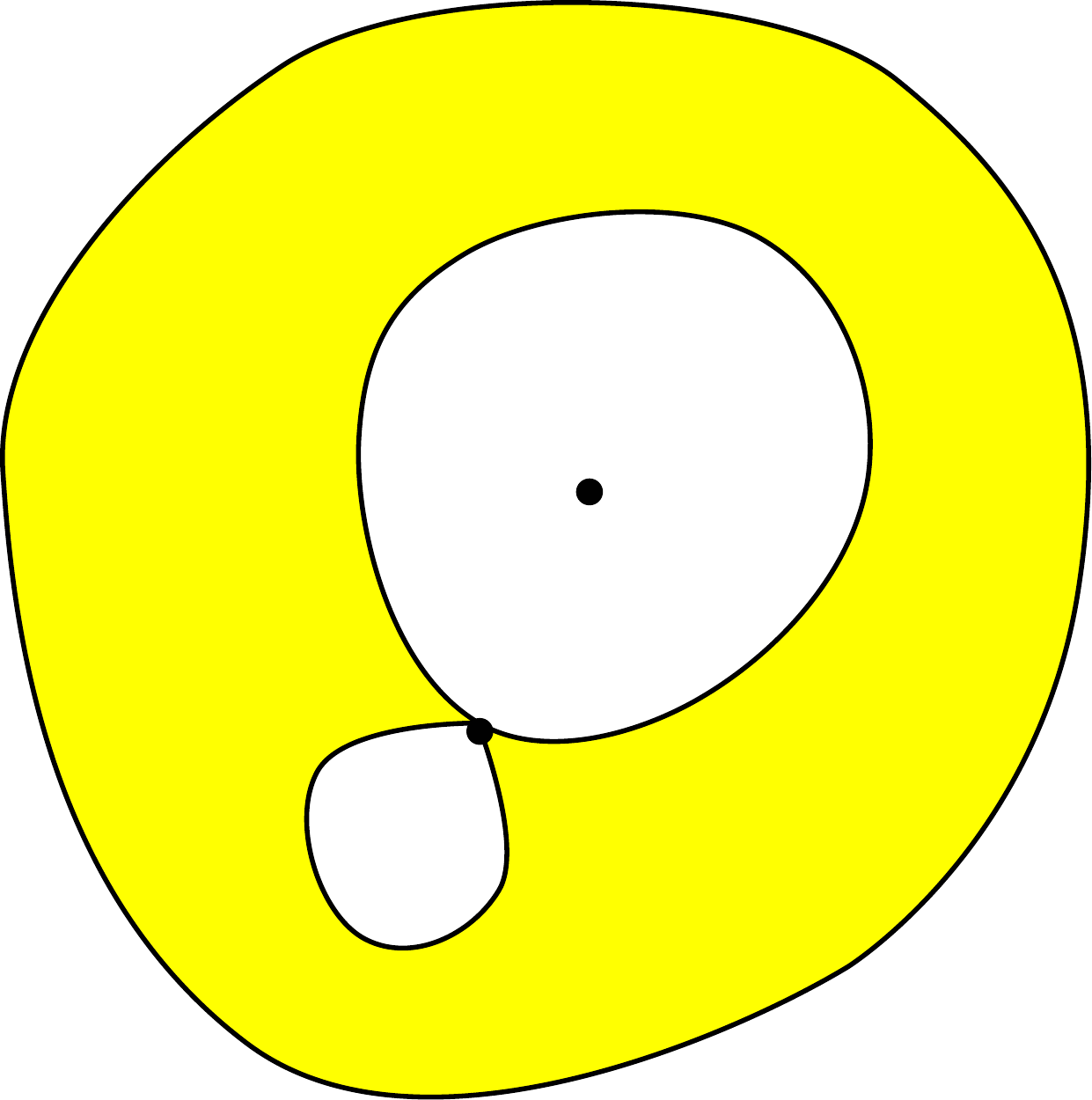}
    \hspace{0.1in}
    \subfigure{
    \def\svgwidth{150pt}
    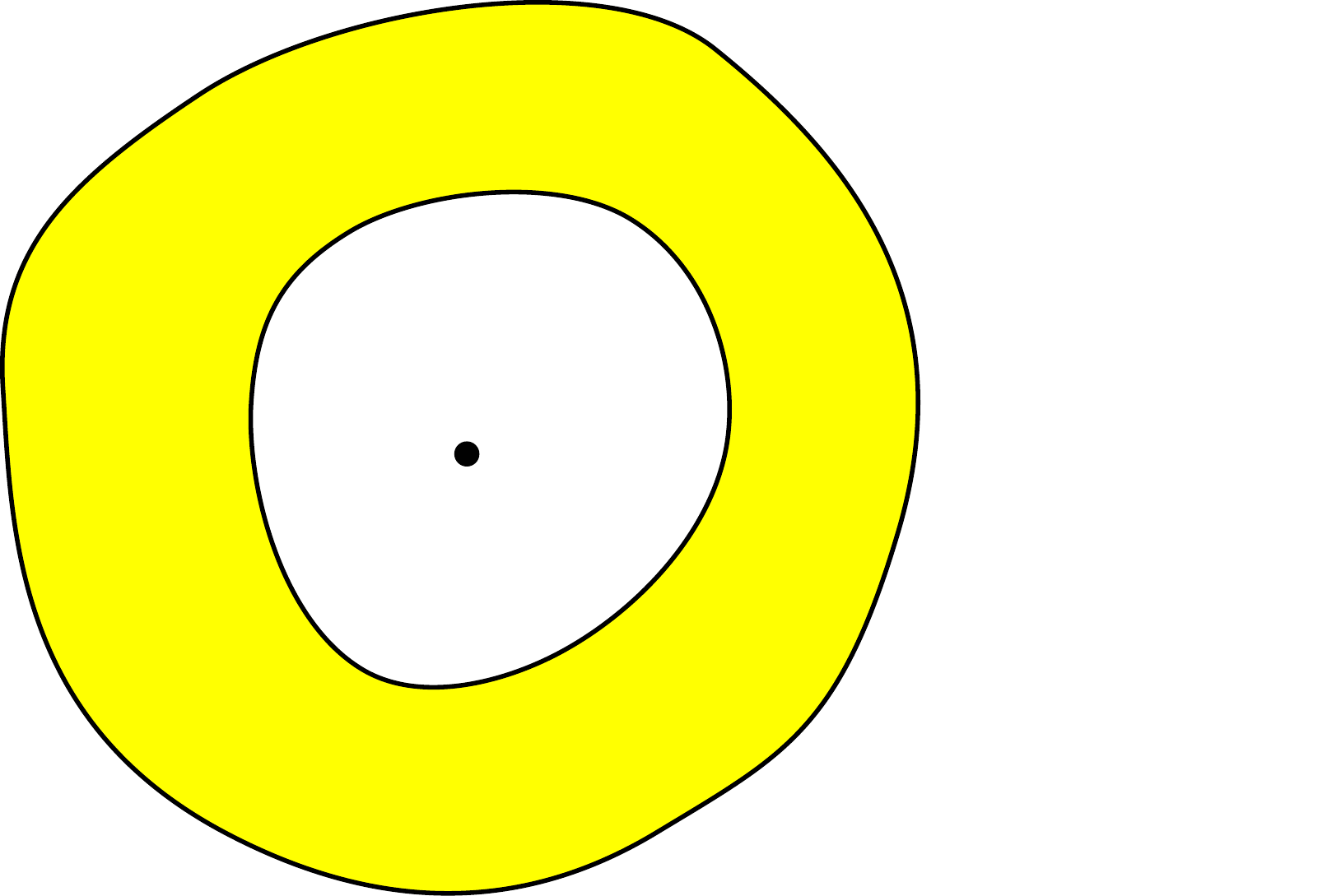}

    \caption{\small Scheme of the two possible configurations for a preimage $U_c$ of $A_0$ such that $c_-\in\partial U_c$. }

    \label{esquemapunxat}
\end{figure}

\begin{rem}\label{remconfpinching}
It follows from the two possible configurations described in Proposition~\ref{confpinching} that if a preimage $A_{\Delta}(a,\lambda)$ surrounds $z=0$, then it will permanently surround $z=0$ when we move the parameter $\lambda$. Indeed, the possible problems can only occur when the critical point $c_-(a,\lambda)$ crosses it (or an iterated image of it). If we move $\lambda$ close to a parameter which corresponds  to the second case of Proposition~\ref{confpinching}, an extra preimage $A'(a,\lambda)$, which does not surround $z=0$, may separate from $A_{\Delta}(a,\lambda)$ or merge with it. However, the set $A_{\Delta}(a,\lambda)$ surrounds continuously $z=0$ when we move $\lambda$.
\end{rem}

We finish this section investigating  the ordering of the sets $A_{\Delta}(a,\lambda)$ as in Lemma~\ref{deforderwell}. A sequence $\Delta$ denotes a unique multiply connected preimage $A_{\Delta}$ of $A_0$ with dynamics described by $\Delta$. We can stablish an order between these sets which relates  their sequences and their relative position in the dynamical plane. This ordering is introduced in the following definition. To simplify notation, it also takes into consideration the case $A_{\Delta}=A_n$ as in Corollary~\ref{anellsacumulen}.

\begin{defin}\label{defordering}
 We say that $\Delta_1\prec\Delta_2$ if the Fatou component $A_{\Delta_2}$ surrounds $A_{\Delta_1}$. We say that $\Delta\prec n$  (or  $n\prec\Delta$)  if $A_n$ surrounds $A_{\Delta}$ (or if $A_{\Delta}$ surrounds $ A_n$). We say that $n\prec m$ if $A_m$ surrounds $A_n$.
\end{defin}

\begin{rem}
From Corollary~\ref{anellsacumulen} we know that $A_{n+1}$ surrounds $A_{n}$. Therefore,  $n\prec m$ if and only if $n<m$.

\end{rem}

   It follows from Proposition~\ref{confpinching} that the order introduced in Definition~\ref{defordering} is consistent and cannot vary with $\lambda$ (c.f.~Remark~\ref{remconfpinching}). It can also be proven that it does neither depend on $a$. We can order the different sequences using their 0's and 1's.  The next lemma explains how the number of 0's at the begining of the sequence influences the order, which is the only thing that we use in this paper (see Theorem~\ref{thmb0}). The general case is  more complicated.

\begin{lemma}\label{ordering}
Let $\Delta$ be a finite sequence of 0's and 1's which contains, at least, one symbol 1. Assume that $\Delta$ begins with $n_{\Delta}$ 0's followed by a 1 (if $\Delta=\{1,\cdots\}$, then $n_{\Delta}=0$). Then, the following hold.

\begin{itemize}
\item If $n_{\Delta}=0$, then $\Delta \prec 0$.
\item If $n_{\Delta}>0$, then $n_{\Delta}-1\prec \Delta \prec n_{\Delta}$.
\end{itemize} 

\end{lemma}
\proof
Let $\Delta=\{i_0,\cdots,i_{p_{\Delta}-1}\}$. If $n_{\Delta}=0$ then, by definition, $A_{\Delta}$ belongs to the annulus $\mathcal{A}_{in}$, which  is surrounded by $A_0$. Hence, we have that $\Delta\prec 0$. Let $\Delta_1=\{0,i_0,\cdots,i_{p_{\Delta}-1}\}$. Then $n_{\Delta_1}=1$, $A_{\Delta_1}\subset \mathcal{A}_{out}$, and $B_{a,\lambda}(A_{\Delta_1})=A_{\Delta}$. Let $A$ be the doubly connected region bounded by bounded by $A_0$ and $A_1$. Then, $B_{a,\lambda}^{-1} (\mathcal{A}_{in})\cap\mathcal{A}_{out}\subset A$. We can conclude that $A_{\Delta_1}\subset A$ and, hence, $0\prec\{0,i_0,\cdots,i_{p_{\Delta}-1}\}\prec 1$. Analogously, we can conclude that for any sequence $\Delta'$ with $n_{\Delta'}>0$ we have $n_{\Delta'}-1\prec\Delta'\prec n_{\Delta'}$. 

\endproof

\section{Escaping dynamics: proof of Theorems A and B}\label{sectionThmAB}

The goal of this section is to study the case where the critical point $c_-$ belongs to $A(\infty)$, to show that all possibilities described in Theorem~\ref{thmavell} can take place, and to prove Theorem~A and Theorem~B. If $c_-$ belongs to $A(\infty)$ then it is eventually mapped onto $D_0$ or $A_0$ (see Lemma~\ref{fatouinfinit}). If it is eventually mapped into $D_0$ then it belongs to a simply connected Fatou component and statement a) of Theorem~\ref{thmavell} holds. On the other hand, if $c_-$ is eventually mapped under iteration into $A_0$ then it belongs to a triply connected Fatou component (see Lemma~\ref{fatouinfinit}), which may surround $z=0$ or not. Theorem~A is a direct corollary of Theorem~\ref{thmb0}, Theorem~\ref{thmcaseA}, and Theorem~\ref{thmcaseB}, which stablish the existence of parameters for which these three cases take place. Theorem~B follows almost directly from Theorem~\ref{thmb0}, which establishes the existence of multiply connected hyperbolic  components of parameters $\lambda$ for which $c_-(a,\lambda)\in A_{\Delta}(a,\lambda)$  for every finite sequence $\Delta$ of $0's$ and $1's$  such that $s\prec \Delta$, where $s$ is a natural number which depends on $a$. The set $A_{\Delta}(a,\lambda)$ denotes the preimage of $A_0(a,\lambda)$  introduced in Lemma~\ref{deforderwell}. Notice that here we consider the order of sequences $\Delta$ as in Definition~\ref{defordering}. Before proving Theorem~\ref{thmb0}, Theorem~\ref{thmcaseA} and Theorem~\ref{thmcaseB} we introduce the number $s$.

Fixed $a\in \dis^*$, we can associate a natural number $r(a,\lambda)$ to each $\lambda\in\dis_{\mathcal{C}(a)}^*$. This number $r(a,\lambda)$ is  given by the minimal $n$ such that $c_-(a,\lambda)\in \rm{Bdd}(A_n(a,\lambda))$, where $\rm{Bdd}(A_n(a,\lambda))$ denotes the open region bounded by $A_n(a,\lambda)$ (including $A_n(a,\lambda)$). Since the $A_n(a,\lambda)$ accumulate on $\partial A_{a,\lambda}^*(\infty)$ (see Corollary~\ref{anellsacumulen}), we know that $r(a,\lambda)$ is finite for all $\lambda\in\dis_{\mathcal{C}(a)}^*$. 

\begin{defi}\label{naturalsaux}
Fixed $a\in\dis^*$ and $\lambda\in\dis_{\mathcal{C}(a)}^*$, we define $r(a,\lambda)$ as the minimal $n$ such that $c_-(a,\lambda)\in \rm{Bdd}(A_n(a,\lambda))$. Fixed $a\in\dis^*$ and given $0<\rho<\mathcal{C}(a)$ we define $s(a,\rho):=\rm{max}\{r(n,\lambda) \; | \; |\lambda|=\rho\}$ and $t(a,\rho):=\rm{min}\{r(a,\lambda)\; | \;  |\lambda|\leq \rho, \; \lambda\neq 0\}$.
\end{defi}

 Notice that $s(a,\rho)$ is well defined and finite since the set of parameters $\lambda$ such that $|\lambda|=\rho$ is closed. On the other hand $t(a,\rho)$ is well defined since we are taking a minimum on the set of natural numbers.

\begin{lemma}\label{convst}
If $\rho\rightarrow 0$, then  $s(a,\rho)\rightarrow \infty$ and  $t(a,\rho)\rightarrow \infty$.
\end{lemma}
\proof
The result follows directly from the fact that, as $\lambda\rightarrow 0$, the maps $B_{a,\lambda}(z)$ converge uniformly to $B_a(z)$ (Equation~(\ref{eqblas})) on compact subsets of $\com\setminus\dis_{\epsilon}$, where $\epsilon$ is arbitrarily small. Hence, as $\lambda\rightarrow 0$ the critical point $c_-(a,\lambda)$ of $B_{a,\lambda}$ converges to the critical point $c_-(a)$ of $B_a$. The critical point $c_-(a)$ belongs to the boundary of the maximal domain of definition of the Böttcher coordinate of the superattracting fixed point $z=0$ of $B_a$ (see e.g.\ \cite{Mi1}). In particular, the orbit of $c_-(a)$ converges to $z=0$ but $B_a^n(c_-(a))\neq 0$ for every natural number $n$. Moreover, no point of the maximal domain of definition of the Böttcher coordinate is eventually mapped under iteration of $B_a$ onto $z=0$. It follows from this and the uniform convergence of $B_{a,\lambda}$  to $B_a$ on compact subsets of $\com\setminus\dis_{\epsilon}$ that the annulus $A_0(a,\lambda)$ shrinks to $z=0$ as $\lambda\rightarrow 0$.
We conclude that, as $\lambda\rightarrow 0$, the critical point $c_-(a,\lambda)$ requires an increasing number of iterates to be mapped into the region $\rm{Bdd}(A_0(a,\lambda))$. Consequently, for all $\epsilon>0$ there is a natural number $n(\epsilon)$ such that if $|\lambda|\leq\epsilon$ then $r(a,\epsilon)\geq n(\epsilon)$. Moreover, $n(\epsilon)\rightarrow \infty$ as $\epsilon\rightarrow 0$. This implies that $s(a,\rho)$ and $t(a,\rho)$  converge to $\infty$ as $\rho\rightarrow 0$.
\endproof

We can now start the proof of Theorem B. To do so, we to study the parameters for which the critical point $c_-(a,\lambda)$ belongs to a preimage of $A_0(a,\lambda)$ surrounding $z=0$. It follows from Theorem~\ref{thmavell} that these parameters are precisely the ones for which the singular perturbations $B_{a,\lambda}$ have Fatou components of arbitrarily large finite connectivity.
Let $A_{\Delta}(a,\lambda)$ denote the preimage of $A_0(a,\lambda)$ surrounding $z=0$ with dynamics described by  the sequence $\Delta=\{i_0,\cdots,i_{p_{\Delta}-1}\}$ as introduced in Lemma~\ref{deforderwell}. We consider order of sequences $\Delta$ as in Definition~\ref{defordering}. By definition of $t(a,\rho)$,  there can be no parameter $\lambda$ with $|\lambda|\leq\rho$ such that $c_-(a,\lambda)\in A_{\Delta}(a,\lambda)$  if $\Delta\prec t(a,\rho)$. In the following theorem  we prove that if $s(a,\rho)\prec \Delta$ then there is a multiply connected hyperbolic component $\Omega_{\Delta}$ of parameters such that $c_-(a,\lambda)\in A_{\Delta}(a,\lambda)$. Moreover, $\Omega_{\Delta}$ is contained in the disk of parameters $\dis_{\rho}$.

\begin{thm}\label{thmb0}
Let $0<\rho<\mathcal{C}(a)$. Then, for every sequence $\Delta$ such that $s(a,\rho)\prec \Delta$, there is a multiply connected hyperbolic component $\Omega_{\Delta}\subset\dis_{\rho}^*$ which surrounds the parameter $\lambda=0$ and such that if $\lambda\in\Omega_{\Delta}$ then $c_-(a,\lambda)\in A_{\Delta}(a,\lambda)$.
\end{thm}
\proof

Fix $\rho<\mathcal{C}(a)$ and $\lambda\in\com$ with $|\lambda|=\rho$. Let $\Delta$ be a finite sequence of 0's and 1's such that $s(a,\rho)\prec \Delta$. By definition of $s(a,\rho)$ we have that $c_-(a,\lambda)\in\rm{Bdd}(A_{s(a,\rho)}(a,\lambda))$. Moreover,  by definition of the ordering of the sequences $\Delta$ (see Definition~\ref{defordering}) we have that the multiply connected domain $A_{\Delta}(a,\lambda)$ is not contained in $\rm{Bdd}(A_{s(a,\rho)}(a,\lambda))$. Indeed $A_{\Delta}(a,\lambda)$ surrounds $\rm{Bdd}(A_{s(a,\rho)}(a,\lambda))$. The idea of the proof is the following.  We decrease continuously $\lambda$. Then, the set $A_{\Delta}(a,\lambda)$ shrinks with $\lambda$, obtaining parameters for which $c_-(a,\lambda)\notin \rm{Bdd}(A_{\Delta} (a,\lambda))$. Finally, by continuity there is some parameter $\lambda$ for which $c_-(a,\lambda)\in A_{\Delta}(a,\lambda)$.

Let $p_{\Delta}$ be such that $B^{p_{\Delta}}_{a,\lambda}(A_{\Delta}(a,\lambda))=A_0(a,\lambda)$. The arguments that we use are based on the fact  $A_{\Delta}(a,\lambda)$ moves continuously with respect to $\lambda$ except at the parameters for which there is an $n\leq p_{\Delta}$ such that $c_-(a,\lambda)\in\partial B_{a,\lambda}^n(A_{\Delta}(a,\lambda))$ (see Corollary~\ref{continuitatanells}, Proposition~\ref{confpinching} and Remark~\ref{remconfpinching}). At these parameters, either one of the connected components of the boundary of $A_{\Delta}(a,\lambda)$ is pinched or there are extra connected components of $B_{a,\lambda}^{-p_{\Delta}}(A_0(a,\lambda))$ which share boundary points with $A_{\Delta}(a,\lambda)$ (see Proposition~\ref{confpinching} and Figure~\ref{esquemapunxat}). To simplify the arguments it will be convenient to work with a curve surrounding $z=0$ contained in $A_{\Delta}(a,\lambda)$. The annulus $A_0(a,\lambda)$ moves continuously with respect to $\lambda$ by Corollary~\ref{continuitatanells}. Hence, we can take a family of Jordan curves $\gamma_{\lambda}\subset A_0(a,\lambda)$ which surrounds $z=0$ and depends continuously on $\lambda$.
Indeed, it follows from Lemma~\ref{realannulus}  that we can take $\gamma_{\lambda}$ to be the circle centred at $z=0$ with radius $\left(\frac{|\lambda|}{|a|}\right)^{1/5}$. Let $\gamma_{\lambda}^{\Delta}$ be the unique connected component of $B_{a,\lambda}^{-p_{\Delta}}(\gamma_{\lambda})$ which surrounds $z=0$ and is contained in $A_{\Delta}(a,\lambda)$ (see Proposition~\ref{preimcorbes}). Then, $\gamma_{\lambda}^{\Delta}$ moves continuously with respect to $\lambda$ except at the parameters $\lambda_n$  such that there is an $n$, $0\leq n<p_{\Delta}$, for which $c_-(a,\lambda_n)\in B_{a,\lambda_n}^n(\gamma_{\lambda_n}^{\Delta})$.  However, this is not a problem for the proof. 
Indeed, the set $B_{a,\lambda}^{-p_{\Delta}}(\gamma_{\lambda})$ moves continuously with respect to $\lambda$ for all $\lambda\in\dis_{\mathcal{C}(a)}^*$. Moreover, if the points $z=0$ and $c_-(a,\lambda_n)$ belong to the same connected component of $\com\setminus \gamma_{\lambda_n}^{\Delta}$, then they belong to the same connected component of $\com\setminus \gamma_{\lambda}^{\Delta}$ for all $\lambda$ in a neighbourhood of $\lambda_n$. In the remaining part of the proof we modify continuously $\lambda$ to find parameters $\lambda_{\Delta}$ for which  $c_-(a,\lambda_{\Delta})\in\gamma_{\lambda_{\Delta}}^{\Delta}$. 

If $|\lambda|=\rho$ and $s(a,\rho)\prec \Delta$ then, by definition of $s(a,\rho)$, we have  that $c_-(a,\lambda)$ belongs to the connected component of $\com\setminus \gamma_{\lambda}^{\Delta}$ which contains $z=0$. On the other hand, we have by Lemma~\ref{convst} that there is an $\epsilon>0$ such that  $\Delta \prec t(a,\epsilon)$ and, therefore, if $|\lambda|\leq \epsilon$ then $c_-(a,\lambda)$ belongs to the unbounded component of $\com\setminus \gamma_{\lambda}^{\Delta}$. The set $B_{a,\lambda}^{-p_{\Delta}}(\gamma_{\lambda})$ depends continuously on $\lambda$ and contains the curve $\gamma_{\lambda}^{\Delta}$. Moreover, fixed $\lambda_0$ with $|\lambda_0|<\mathcal{C}(a)$, we know that if a connected component of $\com\setminus \gamma_{\lambda_0}^{\Delta}$ contains both $z=0$ and $c_-(a,\lambda_0)$, then for all $\lambda$ in a neighbourhood of $\lambda_0$ we have that a connected component of $\com\setminus \gamma_{\lambda}^{\Delta}$ contains both $z=0$ and $c_-(a,\lambda)$. Since $c_-(a,\lambda)$ depends continuously on $\lambda$, if we take any continuous path $\eta$ in the parameter plane joining a parameter $\lambda_{\rho}$ with $|\lambda_{\rho}|=\rho$ and a parameter $\lambda_{\epsilon}$ with $|\lambda_{\epsilon}|\leq\epsilon$, then there is a parameter $\lambda_{\Delta}\in\eta$ such that $c_-(a,\lambda_{\Delta})\in\gamma_{\lambda_{\Delta}}^{\Delta}$. Let $\Lambda_{\Delta}$ be the set of parameters $\lambda$ such that $c_-(a,\lambda)\in\gamma_{\lambda}^{\Delta}$. Then $\Lambda_{\Delta}$ is closed (its complement is open) and separates the set of parameters $\lambda$ with $|\lambda|=\rho$ and the set of parameters $\lambda$ with $|\lambda|\leq \epsilon$. We can conclude that the complement of $\Lambda_{\Delta}$ in the parameter plane is disconnected. Finally, since $\Lambda_{\Delta}$ is a subset of the set of parameters $\Omega'_{\Delta}$ for which $c_-(a,\lambda)\in A_{\Delta}(a,\lambda)$, we conclude that there is (at least) a connected component $\Omega_{\Delta}$ of $\Omega'_{\Delta}$ which is multiply connected, is contained in $\dis_{\rho}^*$, and separates $\lambda=0$ from the set $\lambda$ with $|\lambda|= \rho$. This finishes the proof of the theorem.

\endproof

Now we can prove Theorem~B. It is a direct corollary of Lemma~\ref{convst} and Theorem~\ref{thmb0}.

\begin{proof}[Proof of Theorem B]

 Fix $0<\rho<\mathcal{C}(a)$. Then, it follows from Theorem~\ref{thmb0} that, for every finite sequence $\Delta$ of 0's and 1's  as in Lemma~\ref{deforderwell} such that $s(a,\rho)\prec \Delta$, there is a hyperbolic Fatou component $\Omega_{\Delta}$ which surrounds the parameter $\lambda=0$, is contained in the disk of parameters $\lambda$ with $|\lambda|<\rho$, and such that for all $\lambda\in\Omega_{\Delta}$ we have $c_-(a,\lambda)\in A_{\Delta}(a,\lambda)$.  These hyperbolic Fatou components accumulate on $\lambda=0$ since, by Lemma~\ref{convst}, as we decrease $\rho$ to $0$ the natural $s(a,\rho)$ tends to $\infty$.
 \end{proof}

We continue with the proof of the existence of parameters for which statement a) of Theorem~\ref{thmavell} holds.

\begin{thm}\label{thmcaseA}
Let $a\in(0,1)$. Then there exists $\lambda \in \real^+$, $\lambda<\mathcal{C}(a)$, such that $c_-\in A(\infty)$ and $B_{a,\lambda}$ has only Fatou components of connectivity 1 and 2.
\end{thm}

\proof

It is enough to prove that, given $a\in(0,1)$, there exists $\lambda \in \real^+$ with $\lambda<\mathcal{C}(a)$ such that $B_{a,\lambda}^n(c_-)\in D_0$ for some $n>1$. Then, it follows from the Riemann-Hurwitz formula (see Theorem~\ref{riemannhurwitz}) that $c_-(a,\lambda)$ belongs to a simply connected preimage of $D_0$ and statement a) of Theorem~\ref{thmavell} holds.

The proof of the result is similar to the one of Theorem~\ref{thmb0}. However, in this case we  restrict to real parameters to ensure that $c_-$ is eventually mapped into $D_0$. We begin the proof of the result analysing the real dynamics of the maps $B_{a,\lambda}$ for $a\in(0,1)$ and $\lambda\in \real^+$, $\lambda<\mathcal{C}(a)$. We focus on the dynamics within the interval $(-\infty,z_{\infty})$, where $z_{\infty}=1/\overline{a}$ denotes the only pole of $B_{a,\lambda}$ other than $z=0$. We have that

$$\lim_{z\rightarrow -\infty}B_{a,\lambda}(z)=\lim_{z\rightarrow 0^-}B_{a,\lambda}(z)=\lim_{z\rightarrow 0^+}B_{a,\lambda}(z)=\lim_{z\rightarrow (z_{\infty})^{-}}B_{a,\lambda}(z)=+\infty. $$

We know from Proposition~\ref{zeroscrit} that $B_{a,\lambda}(z)$ has  5 preimages of $z=0$ of the form $\xi(\lambda/a)^{1/5}+o(\lambda^{1/5})$ and  5 critical points of $B_{a,\lambda}(z)$ of the form $-\xi(2\lambda/3a)^{1/5}+o(\lambda^{1/5})$, where $\xi$ denotes a fifth root of the unity and $o(\lambda^{1/5})$ is such that $\lim_{\lambda\rightarrow 0} |o(\lambda^{1/5})|/|\lambda^{1/5}|=0$. Due to the symmetry with respect to the real line we know that a zero $z_5=(\lambda/a)^{1/5}+o(\lambda^{1/5})\in(0,z_{\infty})$ and a critical point $c_5=-\xi(2\lambda/3a)^{1/5}+o(\lambda^{1/5})\in (-\infty,0)$ belong to the real line while the other 4 zeros and critical points which appear near $z=0$ are in $\com\setminus\real$ for $\lambda$ small. We can conclude that the map is decreasing in the intervals $(-\infty, c_5)$ and $(0,c_-)$ and increasing in the intervals $(c_5, 0)$ and $(c_-,z_{\infty})$ (see Figure~\ref{esquemareal}). There is no zero in the interval $(-\infty,0)$, while $z_5\in(0,c_-)$ and $z_0\in(c_-,z_{\infty})$. Moreover, there is a  fixed point $x_1(a,\lambda)\in(z_0,z_{\infty})$ which comes from analytic continuation of the repelling fixed point $x_1(a)=1$ of the unperturbed Blaschke products $B_a$. It follows from the uniform convergence of $B_{a,\lambda}$ to $B_a$ when $\lambda\rightarrow 0$ on compact subsets of $\com\setminus\dis_{\epsilon}$ that, for $\lambda$ small enough, the point $x_1(a,\lambda)$ is repelling and is the only fixed point in the segment $(c_-,z_\infty)$ since this is also true for the Blaschke products $B_a$. Moreover, every point in $(c_-,x_1(a,\lambda))$ converges monotonously onto $x_1(a,\lambda)$ under backwards iteration of $B_{a,\lambda}$. Even if it is not necessary to complete the proof of this result, we want to point out that the previous description holds for all $\lambda$ positive such that $\lambda<\mathcal{C}(a)$. Indeed,  in this case it is possible to adapt the construction of Proposition~\ref{conjblas} so that the conjugation $\Phi\circ\phi$ with a Blaschke products $B_{b,t}$ obtained in the construction preserves the real line.

\begin{figure}[hbt!]
\centering
\def\svgwidth{280pt}
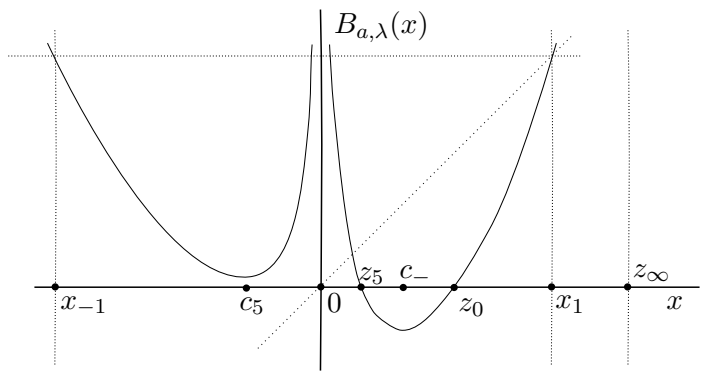
\caption{\small Summary of the real dynamics described in Theorem~\ref{thmcaseA}. }
\label{esquemareal}
\end{figure}

Now consider the zero $z_0(a,\lambda)$ of $B_{a,\lambda}$. Then, it has a sequence of preimages $z_{-n}(a,\lambda)$ such that $B_{a,\lambda}^n(z_{-n}(a,\lambda))=z_0(a,\lambda)$ and the $z_{-n}(a,\lambda)$ accumulate monotonically on $x_1(a,\lambda)$. Moreover, as $\lambda\rightarrow 0$ the zero $z_0(a,\lambda)$ depends continuously on $\lambda$ and converges to the zero $z_0(a)$ of $B_a$. Analogously,  its preimages $z_{-n}(a,\lambda)$ depend constinuously on $\lambda$ and converge uniformly onto preimages $z_{-n}(a)$ of $z_0(a)$ under $B_a$, which accumulate monotonically on the repelling fixed point $x_1(a)$. This allows us to control the sequence $\{z_{-n}(a,\lambda)\}_{n\in\nat}$ of preimages of $z_{0}(a,\lambda)$ as we move continuously $\lambda$ to 0.

To finish the prove we will find an $m\in\nat$ and an interval of parameters $[\Lambda_1,\Lambda_2]\subset\real^+$ such that if $\lambda=\Lambda_2$ then $B_{a,\lambda}^{m+2}(c_-(a,\lambda))=x_{1}(a,\lambda)$ and if $\lambda=\Lambda_1$ then $B_{a,\lambda}^{m+2}(c_-(a,\lambda))<x_1(a,\lambda)$. It would then follow from the fact that the points $z_{-n}(a,\lambda)$ accumulate on $x_1(a,\lambda)$ and depend continuously on $\lambda$ that for some parameters $\lambda\in(\Lambda_1,\Lambda_2)$  (indeed, infinitely many) the point $c_-(a,\lambda)$ is mapped under $B_{a,\lambda}^{m+2}$ onto a preimage $z_{-n}(a,\lambda)$ of $z_0$.  For these parameters the critical point $c_-(a,\lambda)$ belongs to a simply connected Fatou component which is a preimage of $D_0$ and we are done.

Before proving this we comment very briefly the dynamics of the boundary of $A_0$ and its preimages. It follows from the real dynamics described previously (see Figure~\ref{esquemareal}) that $B_{a,\lambda}(\partial T_0\cap\real) =x_1(a,\lambda)$. In particular, $\partial A_0\cap\real$ and all its preimages are also eventually mapped onto $x_1(a,\lambda)$ and move continuously with respect to $\lambda$.

Finally, let $\lambda_0\in\real^+$, $\lambda_0<\mathcal{C}(a)$. Then, there exists an $n\in\nat$ such that $B_{a,\lambda_0}^n(c_-(a,\lambda_0))\notin A_0(a,\lambda_0)$ and for some $\lambda_n<\lambda_0$ we have $B_{a,\lambda_n}^{n}(c_-(a,\lambda_n))\in A_0(a,\lambda_n)$. Indeed, it is enough to consider an $n$ such that the set $A_n(a,\lambda_0)$, as introduced in Corollary~\ref{anellsacumulen}, surrounds $c_-(a,\lambda_0)$. Then we can decrease $\lambda$ continuously until we find a parameter $\lambda_n$ such that $c_-(a,\lambda_n)\in A_n(a,\lambda_n)$ as in the proof of Theorem~\ref{thmb0}. In particular, we can find an interval of parameters $[\Lambda_1,\Lambda_2]$ such that if $\lambda=\Lambda_2$ then $c_-(a,\lambda)\in\partial A_n(a,\lambda)$ and if $\lambda\in[\Lambda_1,\Lambda_2)$ then $c_-(a,\lambda)\notin \overline{ A_n(a,\lambda)}$ but is arbitrarily close to $A_n(a,\lambda)$. In particular, $B_{a,\Lambda_1}^{n+2}(c_-(a,\Lambda_1))<x_1(a,\Lambda_1)$ and $B_{a,\Lambda_2}^{n+2}(c_-(a,\Lambda_2))=x_1(a,\Lambda_2)$. This completes the proof of the theorem.

\endproof

We finish the proof of Theorem~A proving the existence of parameters for which statement b) of Theorem~\ref{thmavell} holds.

\begin{thm}\label{thmcaseB}
Let $a\in\dis^*$ fixed and $\lambda_0\in\dis_{\mathcal{C}(a)}^*$, such that $c_-(a,\lambda_0)$ belongs to an iterated preimage of $D_0(a,\lambda_0)$. Then the parameter $\lambda_0$ is surrounded, within  $\dis_{\mathcal{C}(a)}$, by hyperbolic components of parameters $\lambda$ for which $c_-(a,\lambda)$ belongs to an iterated preimage of $A_0(a,\lambda)$ which does not surround $z=0$. Moreover, if $\lambda_0$ is surrounded by a hyperbolic component $\Omega\subset \dis_{\mathcal{C}(a)}$, then these hyperbolic components are multiply connected.
\end{thm}

\begin{rem}

In Theorem~\ref{thmcaseB}  we say that $\lambda_0$ is surrounded within $\dis_{\mathcal{C}(a)}$. By this we mean the following. Let $\Omega$ be the hyperbolic component which contains $\lambda_0$ and let $\lambda_1\in\dis_{\mathcal{C}(a)}^*\setminus \overline{\Omega}$. Then,  there are hyperbolic components $\Omega'$ such that   any path contained in $\dis_{\mathcal{C}(a)}$ which joins $\lambda_0$ and $\lambda_1$   has non-empty intersection with $\Omega'$. 
\end{rem}

\begin{rem}
We  want to point out that the condition that $\lambda_0$ is surrounded by hyperbolic components $\Omega\subset\dis_{\mathcal{C}(a)}$ is essentially technical. We conjecture that any such $\lambda_0$ is surrounded by a hyperbolic component $\Omega_{\Delta}$ as in Theorem~\ref{thmb0} (see Figure~\ref{param05i}). This condition is somehow equivalent to the condition $s(a,\rho)\prec \Delta$ of Theorem~\ref{thmb0}. They guarantee that the hyperbolic components are contained in $\dis_{\mathcal{C}(a)}^*$. This is important since all proven results depend on the validity of the conclusions of Theorem~\ref{thmcritzeros}, which has as hypothesis $\lambda\in\dis_{\mathcal{C}(a)}^*$. We conjecture that doubly connected hyperbolic components $\Omega_{\Delta}$ accumulate on the boundary of the set of parameters where the conclusions of Theorem~\ref{thmcritzeros} hold.

\end{rem}

\begin{proof}[Proof of Theorem~\ref{thmcaseB}]

The proof is analogous to the one of Theorem~\ref{thmb0}.  Let $a\in\dis^*$ and  $\lambda_0\in\dis_{\mathcal{C}(a)}^*$ such that $c_-(a,\lambda_0)$ belongs to an iterated preimage $D(a,\lambda_0)$ of $D_0(a,\lambda_0)$. Since it contains a critical point, $D(a,\lambda_0)$ is mapped 2-1 onto its image. In particular, it contains 2 preimages of the zero $z_0(a,\lambda)$, say $w^1(a,\lambda)$ and $w^2(a,\lambda)$. When we move the parameter $\lambda$ so that $c_-(a,\lambda)$ exists the set $D(a,\lambda)$, this Fatou component continuously splits into the disjoint union of two preimages of $D_0(a,\lambda)$, say $D^1(a,\lambda)$ and $D^2(a,\lambda)$. The sets $D^1(a,\lambda)$ and $D^2(a,\lambda)$ contain $w^1(a,\lambda)$ and $w^2(a,\lambda)$, respectively.

We know from Corollary~\ref{anellsacumulen} that there are preimages $A_n(a,\lambda_0)$ of $A_0(a,\lambda_0)$ which accumulate on $\partial A_{a,\lambda_0}^*(\infty)$. As a consequence, there are preimages of $A_0(a,\lambda_0)$ which accumulate on any preimage of $\partial A_{a,\lambda_0}^*(\infty)$. In particular, we can take a  multiply connected Fatou component $A(a,\lambda_0)$  which surrounds $D(a,\lambda_0)$ but does not surround $z=0$ (see Figure~\ref{esquemabounded} (left)). Since the free critical point $c_-(a,\lambda_0)$ does not belong to any preimage of $A_0(a,\lambda_0)$, it follows from the Riemann-Hurwitz formula (see Theorem~\ref{riemannhurwitz}) that  $A(a,\lambda_0)$ is indeed doubly connected (compare with the proof of Lemma~\ref{fatouinfinit}). Let $\rm{Bdd(A(a,\lambda_0))}$ denote the set of points bounded by $A(a,\lambda_0)$.  Since $\partial D(a,\lambda_0)$ is preperiodic, we may choose $A(a,\lambda_0)$ so that there is no annulus $A'(a,\lambda_0)\subset \rm{Bdd(A(a,\lambda_0))}$ such that $B^n_{a,\lambda_0}(A(a,\lambda_0))=A'(a,\lambda_0)$ for some $n>0$.

Let $\gamma_{\lambda}\subset A_0(a,\lambda)$ be any Jordan curve which surrounds $z=0$ and depends continuously on $\lambda$ (we can take it since, by Proposition~\ref{contboundary},  $A_0(a,\lambda)$ depends continuously on $\lambda$). Then $\gamma_{\lambda_0}$ has a unique preimage in $A(a,\lambda_0)$, say $\gamma'_{\lambda_0}$. Moreover $\gamma'_{\lambda_0}$ is a Jordan curve. Assume that we can take a closed curve in the parameter plane $\eta:\left[0,1\right]\rightarrow \dis_{\mathcal{C}(a)}$ such that $\eta(0)=\lambda_0$, that $c_-(a,\eta(y))\in\rm{Int}(\gamma'_{\eta(y)})$ for all $y\in (0,1)$, and that $c_-(a,\eta(1))\in\gamma'_{\eta(1)}$.
It follows from the fact that there is no image of $A'(a,\lambda_0)$ in $ \rm{Bdd(A(a,\lambda_0))}$ that $\gamma'_{\eta(y)}$ is a Jordan curve which moves continuously with respect to $y$ for all  $y\in\left[0,1\right)$. At the parameter $\eta(1)$, the curve $\gamma'_{\eta(1)}$ consists of the union of two Jordan curves, $\gamma^1$ and $\gamma^2$, which have the common point $c_-(a,\eta(1))$. Each of the curves $\gamma^1$ and $\gamma^2$ surrounds one of the two Fatou components $D^1(a,\lambda)$ and $D^2(a,\lambda)$ in which the set $D(a,\lambda)$ splits after $c_-(a,\lambda)$ exists it since $B_{a,\lambda}(\rm{Int}(\gamma^1))=B_{a,\lambda}(\rm{Int}(\gamma^2))$.
 Moreover, for all $y\in\left[ 0, 1\right]$,  the curve $\gamma'_{\eta(y)}$ does not surround $z=0$. By Lemma~\ref{fatouinfinit} we know that $A(a,\eta(1))$ is triply connected. The curves $\gamma^1$ and $\gamma^2$ bound two connected components of $\partial A(a,\eta(1))$. Since they do not surround $z=0$, the third connected component of  $\partial A(a,\eta(1))$ is necessarily a Jordan curve which does not surround $z=0$.
 We conclude that $A(a,\eta(1))$ is a triply connected Fatou component which contains the point $c_-(a,\eta(1))$ and does not surround $z=0$ (see Figure~\ref{esquemabounded} (right), c.f.\ Figure~\ref{dynamfigureABC} (d)).

\begin{figure}[hbt!]
    \centering
   
    \def\svgwidth{250pt}
    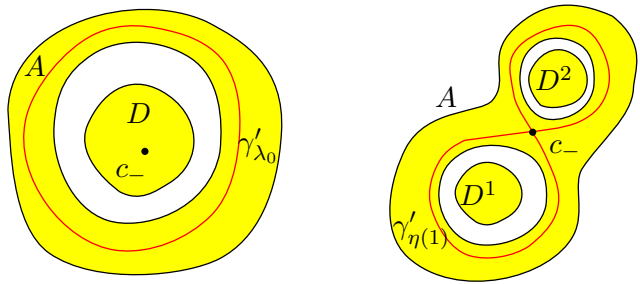

    \caption{\small Scheme of the  situation described in the proof of Theorem~\ref{thmcaseB}. We draw in red the curve $\gamma_{\eta(y)}^1$ for $\eta(0)=\lambda_0$ (left) and $\eta(1)$ (right). }

    \label{esquemabounded}
\end{figure}

To finish the proof  we need to find the curve of parameters $\eta$. Take a parameter $\lambda_{out}$  such that $c_-\notin \overline{ D(a,\lambda_{out})}$ (for example, take $\lambda_{out}$ in some hyperbolic component $\Omega_{\Delta}$ as in Theorem~\ref{thmb0}). By redefining the set $A(a,\lambda)$, we may assume that $A(a,\lambda_{out})$ does not bound the critical point $c_-(a,\lambda_{out})$. Then given any curve $\zeta$ on the parameter plane joining $\lambda_0$ and $\lambda_{out}$,  there is a curve $\eta\subset \zeta$ satisfying the previous conditions. Finally, following the same reasoning used in the proof of Theorem~\ref{thmb0}, we can conclude that there is a hyperbolic component $\Omega$ which surrounds $\lambda_0$ in $\dis_{\mathcal{C}(a)}$ such that if $\lambda\in\Omega$ then $c_-(a,\lambda)\in A(a,\lambda)$ and the multiply connected Fatou component $A(a,\lambda)$ does not surround $z=0$. Moreover, if $\lambda_{out}$ belongs to a multiply connected hyperbolic component $\Omega_{out}\subset \dis_{\mathcal{C}(a)}$ which surrounds $\lambda_0$, then $\Omega$ is a multiply connected hyperbolic component which separates $\Omega_{out}$ and $\lambda_0$.

\end{proof}

\bibliography{bibliografia}
\bibliographystyle{AIMS}

\end{document}